\documentclass[reqno]{amsart}
\usepackage{amssymb, graphicx, cite, url,color}

\newtheorem{thm}{Theorem}[section]
\newtheorem{cor}[thm]{Corollary}
\newtheorem{lem}[thm]{Lemma}
\newtheorem{prop}[thm]{Proposition}
\theoremstyle{definition}

\theoremstyle{remark}
\newtheorem{rem}[thm]{Remark}
\numberwithin{equation}{section}

\newcommand{\Real}{\mathbb R}
\newcommand{\eps}{\varepsilon}

\renewcommand{\kappa}{\varkappa}
\newcommand{\cA}{\mathcal{A}}
\newcommand{\cL}{\mathcal{L}}
\newcommand{\bV}{\mathbf{V}}
\newcommand{\bW}{\mathbf{W}}

\newcommand{\dom}{\mathop{\rm dom}}
\newcommand{\cP}{\mathcal{P}}
\newcommand{\rR}{R_\mu}
\newcommand{\ess}[1]{\sigma_{\kern-0.08em  e\kern-0.05em s\kern-0.06em s}(#1)}
\newcommand\rep{\textstyle\frac r\eps}

\newcommand{\rRz}{\mathrm{R}_z}

\begin{document}

\title{Membranes with thin and heavy inclusions: asymptotics of spectra}%
\author{Yuriy Golovaty}%

\begin{abstract}
We study the asymptotic behaviour of eigenvalues  of  2D vibrating systems with  mass density perturbed in a vicinity of closed curves. The threshold case in which the resonance frequencies of the  membrane and the frequencies of thin inclusion coincide is investigated. The perturbed eigenvalue problem can be realized as a family of self-adjoint operators acting on varying Hilbert spaces. However the so-called limit operator is non-self-adjoint and possesses the Jordan chains of length $2$. Apart from the lack of self-adjointness, the operator has non-compact resolvent. As a consequence, its spectrum  has a complicated structure, for instance, the spectrum contains a countable set of eigenvalues with infinite multiplicity. The complete asymptotic analysis of  eigenvalues   has been  carried out.
\end{abstract}

\address{Dept. of Mechanics and Mathematics, Ivan Franko National University of Lviv, Universytetska str.1, Lviv, 79000, Ukraine}%
%

\maketitle

\section{Introduction and Statement of  Problem}

The mechanical systems with strongly inhomogeneous mass distributions  have become the subject of intensive experimental and theoretical studies
since the time of Poisson and Bessel  \cite[Ch.2]{Sarpkaya2010}, and  a lot of research has been devoted to the analysis of vibrating systems with so-called added masses.
Historically, the ﬁrst relevant mathematical models in classical mechanics go back to the first half of the 20th century (see e.g. \cite{GantmakherKrein1950} and the references given there). Many authors have investigated  properties of strings and rods with the mass densities perturbed by  finite or infinite
sums $\sum M_k\delta(x-x_k)$, where $\delta$ is Dirac's delta-function and $M_k$ is an added mass at the point $x_k$.
Recently,  such models in dimensions two and three with heavy inclusions of a different geometry are widely used not only in mechanics, but in various fields of science and technology such as physics of liquid crystals, physical chemistry of polymers, micelles and microemulsions, molecular theory, cell membrane theory  \cite{BatesFredrickson1990, SensTurner1997, PratibhaParkSmalyukh2010}.  For instance,  cell membranes are known to contain  embedded proteins and various colloidal particles \cite{LadbrookeChapman1969}.

In higher dimensions, the perturbation of mass densities by the $\delta$-functions often leads to incorrect mathematical models, because the formal differential equations which appear have no mathematical meaning. As an example of such ill-posed problem we can  consider the eigenvalue problem for the Laplace operator
\begin{equation*}
  -\Delta u=\lambda (1+M\delta(x))u \text{ \ in }\Omega, \qquad u=0 \text{ \ on }\partial\Omega,
\end{equation*}
where $\Omega$ is a bounded domain in $\Real^3$ containing the origin. The equation has no non-trivial solution, because any such solution $u$ has a singularity at $x=0$ and therefore the product $\delta(x)u(x)=u(0)\delta(x)$ is not defined. The new and at the same time obvious idea was instead to replace the $\delta$-function with its regularization $\eps^{-3}q(x/\eps)$, where $q$ is a  function of compact support, and study the asymptotic behaviour of eigenvalues and eigenfunctions as $\eps\to 0$.
The problem was first investigated by E. S\'{a}nchez-Palencia \cite{MyFirstPaperOfSP, PalenciaHubertBook}, who proved the existence of the so-called local eigenvibrations: the eigenfunctions are significant in a small neighborhood of the origin only.

The model in which the density is  perturbed by $\sum M_k\delta(x-x_k)$  is  not adequate even in the one-dimensional case, when dealing with the large
masses $M_k$. The very heavy inclusions cause a strong local reaction of  vibrating system, but  this phenomenon  can not be described on the discrete set which is a support of the sum of Dirac's functions. The geometry of  small domains where the large  masses are loaded should also have an effect on the form of  eigenvibrations. In \cite{GolovatyNazarovOleinikSoboleva1988},   asymptotic analysis was applied to a spectral problem for the Sturm-Liouville operator with  weight function  of the form
$\rho_\eps(x)=\rho(x)+\eps^{-m}q(x/\eps)$,
where $q$ is a function of compact support and $m\in\Real$. For the case $m=1$ the perturbation is a $\delta$-like sequence, but the most interesting cases of the limit behaviour of eigenfunctions as $\eps\to 0$ are those when the power $m$ is  greater than $1$.

These advanced  models have attracted considerable attention in  the mathema\-ti\-cal literature over the last three  decades (see e.g. \cite{LoboPerez2003} for a review).
The spectral properties of differential operators  with  weight functions having the form
\begin{equation*}
  \rho_\eps(x)=\rho(x)+\sum\eps^{-m_k}q_{k,\eps}(x),
\end{equation*}
where $q_{k,\eps}$ are compactly supported in vicinity of different sets,
have been investigated in numerous articles. We mention here
\cite{GolovatyjNazarovOleinik1990} for the Laplace operator in dimension $3$, \cite{GolovatyTrudy1992, GolovatyLavrenyuk2000} for ordinary differential operators of the fourth order and the biharmonic operator, \cite{MelnikNazarov2001, Melnik2001, ChechkinMelnyk2012} for boundary value problems on junctions of a very complicated geometry,  \cite{GolovatyHrabchak2007, GolovatyHrabchak2010} for the Sturm-Liouville operators on metric graphs.
The spectral properties of strings with rapidly oscillating and periodic densities have been treated in \cite{OleinikShamaevYosifianBook, CastroZuazuaSIAM2000, CastroZuazuaEJAM2000} in the framework of homogenization theory. Another model in which the heavy inclusions were regarded as rigid ones has been studied in \cite{Rybalko2002}.
Since the $90$s of the last century, a series of papers  was  published  concerning  2D and 3D elastic systems with many concentrated masses near the boundary
\cite{LoboPerez1993, LoboPerezSA1995, LoboPerez1995, Chechkin2004, Chechkin2005, NazarovPerez2009, ChechkinPerezYablokova2005, Yablokova2005}. New asymptotic results
for the spectral problems in domains surrounded by thin stiff and heavy bands, when the mass density and  stiffness are simultaneously perturbed in a neighbourhood of the boundary,
were obtained in  \cite{GomezNazarovPerez2006, GomezNazarovPerez2006.1, NazarovPerez2018, GomezNazarovPerez2021}. The Neumann eigenvalue problem for a membrane, almost the entire mass of which is concentrated around the boundary, was studied in \cite{RivaProvenzano2018}.

\begin{figure}[t]
  \centering
  \includegraphics[scale=0.5]{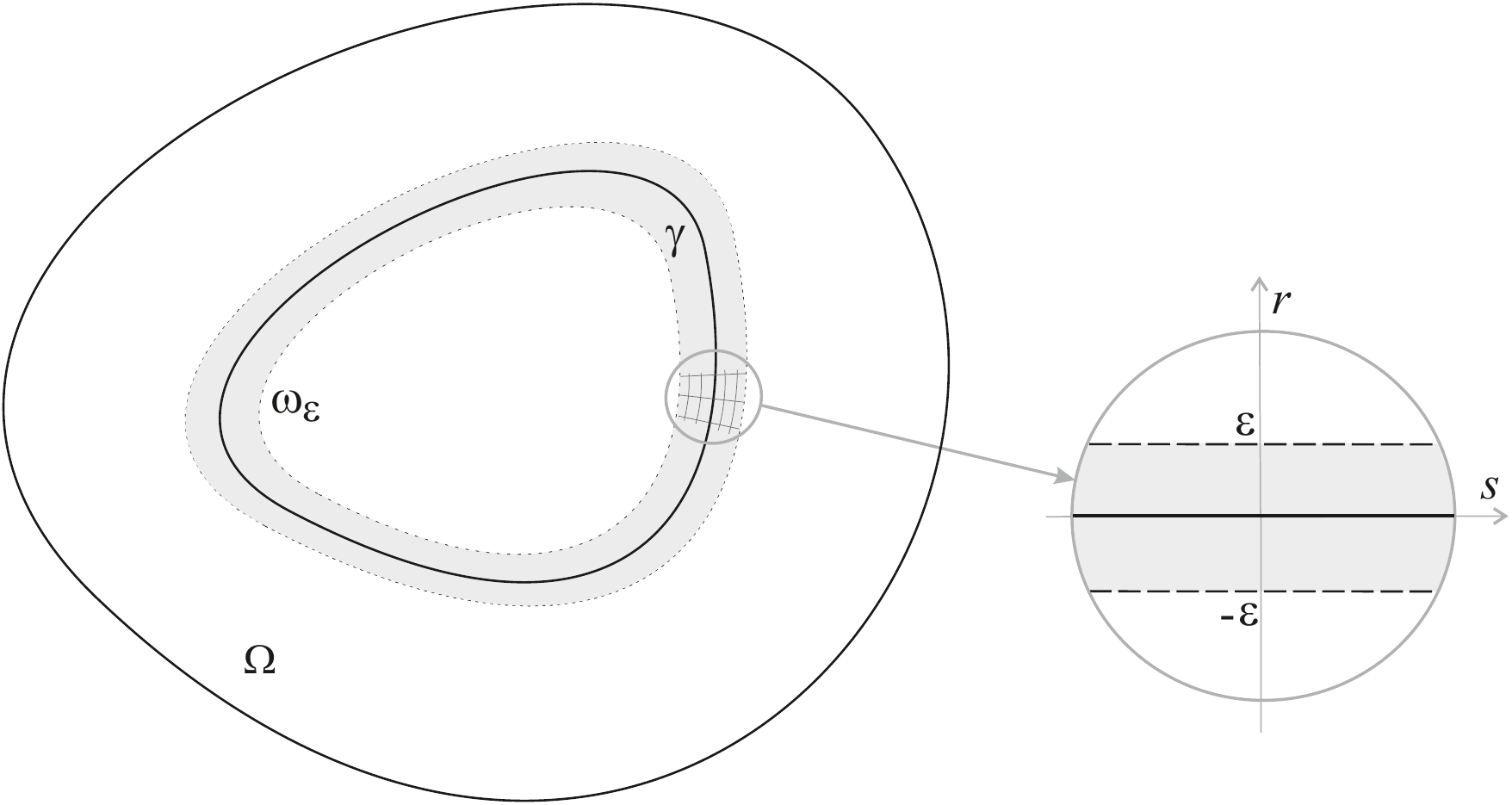}\\
  \caption{Membrane with  heavy and thin inclusion}\label{FigLocalCoords}
\end{figure}

In this paper we study the eigenvibration characteristics of a membrane with  heavy and thin inclusion inside.
Let $\Omega$ be a smooth bounded domain of $\mathbb{R}^2$, and let
$\gamma\subset\Omega$ be a smooth closed curve.  We will denote by $\omega_\eps$ the $\eps$-neighborhood of $\gamma$
(i.e., the union of all open balls of radius $\eps$ around a point on $\gamma$).
Consider $\eps>0$ sufficiently small such that $\omega_\eps\subset\Omega$ and the boundary $\partial \omega_\eps$ is smooth.  Assume  $\rho$ is a smooth uniformly positive functions in $\Omega$, and
\begin{equation}\label{DensityP}
  \rho_\eps(x)=
  \begin{cases}
    \rho(x), & \text{if } x\in \Omega\setminus \omega_\eps, \\
    \eps^{-2}q_\eps(x), & \text{if } x\in \omega_\eps.
  \end{cases}
\end{equation}
To specify   explicit dependence of $q_\eps$ on  $\eps$ we introduce  the Fermi normal coordinates  in $\omega_\eps$  (see Fig.~\ref{FigLocalCoords}).
Let $\alpha\colon [0,|\gamma|)\to \Real^2$ be the unit-speed smooth parametrization of $\gamma$ with the natural parameter $s$, and $|\gamma|$ is the length of $\gamma$. Then the vector $\nu=(-\dot{\alpha}_2, \dot{\alpha}_1)$ is a unit normal on $\gamma$. Set $ x=\alpha(s)+r\nu(s)$ for $(s,r)\in [0,|\gamma|)\times (-\eps, \eps)$,
where $r$ is the signed distance from $x$ to $\gamma$. Suppose that
\begin{equation*}
q_\eps(x)=q({\textstyle\frac r\eps}),
\end{equation*}
where $q$ is a smooth positive function in  $[-1,1]$. We see that the perturbation of mass density varies on the normal direction to $\gamma$ only.
Let us consider the spectral problem
\begin{equation}\label{PEproblem}
  - \Delta u_\eps+au_\eps = \lambda^\eps \rho_\eps u_\eps\quad\text{in }\Omega, \qquad \ell u_\eps=0 \quad \text{on } \partial\Omega,
\end{equation}
where $a\in C^\infty(\overline{\Omega})$ and $\ell v=0$ denotes the Dirichlet, Neumann or Robin boun\-da\-ry conditions on $\partial\Omega$. Our goal is to describe the asymptotic behaviour of the eigenvalues $\lambda^\eps$  of  \eqref{PEproblem} as $\eps\to 0$.

This paper is a continuation of \cite{GolovatyGomezPerezLoboCR2002, GolovatyGomezLoboPerez2004}, where the problem with the Dirichlet boundary condition and more general perturbation $\eps^{-m}q_\eps$ of  mass density has been treated using the variational approach  and the case $m=3$ has been completely investigated by asymptotic methods.
Also, it has been shown that there exist five different limit behaviours for the spectrum and eigenspaces depending on $m$: $m<1$, $m=1$, $1<m<2$, $m=2$ and $m>2$. These cases differ by the form of eigenvibrations and  place on the membrane  where the main part of their ``energy'' is concentrated in the limit.
From the mathematical viewpoint, the difference between the cases is that the spectral parameter in the limit eigenvalue problems appears alternately both in a differential equation on $\Omega$  and in an equation on the strip $\gamma\times (-1,1)$, which is a dilatation of $\omega_\eps$,  and even in coupling conditions on  $\gamma$.

The threshold case $m=2$ is the most difficult and interesting one to analyse, because the spectral parameter is included simultaneously in two differential equations, and the limit spectral problem is associated with  a non-self-adjoint operator.
The main insight of the present work is to exhibit this non-self-adjoint operator,  its spectrum and generalized eigenspaces.

To better understand what happens in dimension $2$, we briefly discuss the similar model for a string. Recently, we have revised some results from \cite{GolovatyNazarovOleinikSoboleva1988} concerning the case $m=2$.  In \cite{Golovaty2020}, we have studied the limiting behavior, as $\eps\to 0$, of eigenvalues
and eigenfunctions  of  the  problem
\begin{gather}\label{PertPrbEq}
        -y_\eps''+Qy_\eps=\lambda^\eps \rho_\eps y_\eps,\quad x\in (a,b),
        \\ \label{PertPrbCondAB}
        y_\eps(a)\cos\alpha+y_\eps'(a)\sin\alpha=0,\quad y_\eps(b)\cos\beta+y_\eps'(b)\sin\beta=0
\end{gather}
with the weight function $\rho_\eps$ given by \eqref{DensityP}, where $\Omega=(a,b)$ is an interval containing the origin, and $\omega_\eps=(-\eps,\eps)$. For each real $\alpha$ and $\beta$ the problem can be associated with  a family of self-adjoint operators $S_\eps$ in the weighted  space $L_2(\rho_\eps, \Omega)$. The operators $S_\eps$ are defined by $S_\eps\phi=\rho_\eps^{-1}(-\phi''+Q\phi)$ on functions $\phi\in W_2^2(\Omega)$ obeying boundary conditions \eqref{PertPrbCondAB}. The spectra of $S_\eps$ are real, discrete and simple.
It is worth noting that the study of operator families acting on  varying  spaces entails some mathematical difficulties. First of all, the question arises how to understand the convergence of such families. Next, if these operators do converge in some sense, does this convergence implies the convergence of their spectra (see, e.g.,  \cite[III.1]{OleinikShamaevYosifianBook}, \cite{MelnikConvergence2001, MugnoloNittkaPost2010}).
By abandoning the self-adjointness, we have realized \eqref{PertPrbEq}, \eqref{PertPrbCondAB} as a family of non-self-adjoint matrix operators $\cA_\eps$ acting on the fixed Hilbert space $L_2(\rho,(a,b))\times L_2( q,(-1,1))$ (see details in \cite{Golovaty2020}). The operators are certainly similar to self-adjoint ones and their  resolvents are compact. Moreover  the spectra of $\cA_\eps$ and $S_\eps$ coincide and the corresponding eigenspaces are isomorphic. It has been proved that $\cA_\eps$ converge in the norm resolvent sense (the resolvents of $\cA_\eps$ converge in the uniform norm as $\eps\to 0$) to the matrix operator $\cA$ associated with  the eigenvalue problem
\begin{align*}
        &-v''+Qv=\lambda \rho v\;\;\text{in }(a,0)\cup(0,b),
        \quad -w''=\lambda  qw\;\; \text{in } (-1,1),
        \\
        &\quad w'(-1)=0,\quad w'(1)=0,\quad v(-0)=w(-1),\quad v(+0)=w(1),
         \\
         &\quad v(a)\cos\alpha+v'(a)\sin\alpha=0,\quad v(b)\cos\beta+v'(b)\sin\beta=0.
\end{align*}
Surprisingly enough,  $\cA$ is not similar to a self-adjoint one, because it possesses multiple eigenvalues with the Jordan chains of length $2$. In view of   \cite[Ch.1]{GohbergKrein1978},\cite{BrambleOsborn1973}, the norm resolvent convergence  $\cA_\eps\to \cA$ implies the ``number-by-number'' convergence of the corresponding eigenvalues and some results on the convergence of eigenspaces.

Problem \eqref{PEproblem} can also be associated with a family of matrix operators acting on the same Hilbert space. However, in two dimensions there is no uniform convergence of resolvents, because the limit operator $\cP$ which can be obtained in the strong resolvent topology has  non-compact resolvent~\cite[Theorem~IV.2.26]{Kato}. It is known that the strong resolvent convergence of operators does not imply the convergence of their spectra. For this reason, we study the convergence of eigenvalues  of \eqref{PEproblem}  without touching on the convergence of corresponding operators. The operator $\cP$ is non-self-adjoint  and possesses nontrivial Jordan chains. In Sec.~\ref{Sect3}, we describe the spectrum of $\cP$  and the generalized eigenspaces. The spectrum is discrete and real; it consists of eigenvalues of three types: \textit{(a)} eigenvalues of finite multiplicity with generalized eigenspaces generated by eigenvectors only; \textit{(b)} infinite-fold eigenvalues  with generalized eigenspaces generated by eigenvectors only; \textit{(c)} infinite-fold eigenvalues  with generalized eigenspaces containing both eigenvectors and generalized eigenvectors.
In fact, this spectrum is a union of  spectra of two operators that correspond to  two different parts of the vibrating system, namely, the membrane clamped along the curve $\gamma$ and the thin heavy inclusion. Using the  method of quasimodes  we describe the asymptotic behaviour of eigenvalues of \eqref{PEproblem} as perturbation of $\sigma(\cP)$. In Sec.~\ref{SectCaseII} and \ref{SectCaseIII}, we prove the existence of a countable number of eigenvalues with the asymptotics $\lambda_{\eps,k}=\lambda+\eps\lambda_1^{(k)}+O(\eps^2)$ as $\eps\to 0$, $k\in \mathbb{N}$, where $\lambda$ is an eigenvalue of type \textit{(b)} or \textit{(c)}. The set of correctors $\{\lambda_1^{(k)}\}_{k=1}^\infty$ is a spectrum of some pseudodifferential operator $N_{\lambda}$ acting in $L_2(\gamma)$. In addition, if $\lambda$ is an  eigenvalue of type (c), there also exists $2d$ eigenvalues with the half-integer power asymptotics $\mu_{\eps,j}^\pm=\lambda\pm\eps^{1/2}\lambda_{1/2,j}+\eps\lambda_{1,j}^\pm+O(\eps^{3/2})$, as $\eps\to 0$, where $d$ is a number of Jordan chains of length $2$  corresponding to $\lambda$. In Sec.~\ref{SectCaseI} we prove that there exist $K$ eigenvalues of \eqref{PEproblem} with the asymptotics  $\lambda_{\eps,k}=\lambda+\eps\lambda_1^{(k)}+O(\eps^2)$, where $\lambda$ is an  eigenvalue of type \textit{(a)} with multiplicity $K$.

\section{Formal construction of the limit operator}

It will be convenient to parameterize the curve $\gamma$ by  points of a circle. It  will allow us not to indicate every time that  functions on $\gamma$  are periodic on $s$.
Let $S$ be the circle of the  length $|\gamma|$. Then $\omega_\eps$ is diffeomorphic to the cylinder $S\times (-\eps, \eps)$. We also set $\omega=S\times (-1, 1)$. Here and subsequently, writing  ``in $\omega$'' after an equation, we mean that the equation is considered in the rectangle  $(0, |\gamma|)\times(-1,1)$ and the corresponding solution is a $|\gamma|$-periodic function on $s$.

Let us denote by $\gamma_t$ the  curve that is obtained from $\gamma$ by flowing for ``time'' $t$ along the normal vector field, i.e.,
  $\gamma_t=\{x\in\Real^2\colon\; x=\alpha(s)+t\nu(s), \; s\in S\}$.
Then the boundary of $\omega_\eps$ consists of two curves $\gamma_{-\eps}$ and $\gamma_{\eps}$.
The domain $\Omega$ is divided by  $\gamma$ into subdomains $\Omega_-$ and $\Omega_+$. Suppose that
$\partial\Omega_-=\partial\Omega\cup\gamma_-$ and $\partial\Omega_+=\gamma_+$, where $\gamma_-$ and $\gamma_+$ are two edges of the cut $\gamma$.
In the sequel,  the  coordinate $r$ increases in the direction from $\Omega^-$ to $\Omega^+$.

For any  $\eps>0$,  problem \eqref{PEproblem} admits a self-adjoint operator rea\-li\-zation in the weighted Lebesgue  space $L^2(\rho_\eps, \Omega)$.
We introduce  operator
$T_\eps$ defined by
\begin{equation*}
T_\eps\phi=\rho_\eps^{-1}(- \Delta\phi+a\phi)
\end{equation*}
on functions $\phi\in W_2^2(\Omega)$ obeying the boundary condition $\ell \phi=0$ on $\partial\Omega$.
Obviously, the spectrum of $T_\eps$ is real and discrete.

We  look for the  approximation to the eigenvalue $\lambda^\eps$ and the corresponding eigenfunction $u_\eps$ of \eqref{PEproblem} in the form
\begin{equation}\label{AsymptExps}
\lambda^\eps\sim\lambda+\cdots,\qquad
u_\eps(x)\sim
    \begin{cases}
        v(x)+\cdots & \text{if } x\in \Omega\setminus \omega_\eps,\\
        w(s,\rep)+\cdots & \text{if } x\in \omega_\eps.
    \end{cases}
\end{equation}
Since  $u_\eps$ solves \eqref{PEproblem} and the domain $\omega_\eps$ shrinks to $\gamma$ as $\eps\to 0$, the function $v$  must be a solution of the equation
$-\Delta v+av=\lambda \rho v$ in $\Omega\setminus\gamma$ that satisfies the boundary condition $\ell v=0$ on $\partial\Omega$. Of course, $v$ must also satisfy  appropriate transmission conditions on $\gamma$.
To find these conditions, we must examine more closely the equation in \eqref{PEproblem} in a vicinity of $\gamma$.

Returning  to the local coordinates $(s,r)$,
we see that the vectors
$ \alpha=(\dot{\alpha}_1, \dot{\alpha}_2)$, $\nu=(-\dot{\alpha}_2, \dot{\alpha}_1)$
give the Frenet frame for $\gamma$.
The Jacobian of transformation
\begin{equation*}
x_1=\dot{\alpha}_1(s)-r\dot{\alpha}_2(s),\quad x_2=\dot{\alpha}_2(s)+r\dot{\alpha}_1(s)
\end{equation*}
 has the form $J(s,r)=1-r \kappa(s)$, where $\kappa=\det(\dot{\alpha},\ddot{\alpha})$ is the signed curvature of $\gamma$.
We see that $J$ is positive for sufficiently small $r$, because  the curvature $\kappa$  is  bounded on $\gamma$.
In addition,   the  Laplace-Beltrami operator becomes
\begin{equation*}
\Delta \phi=J^{-1}\left(\partial_s(J^{-1}\partial_s \phi)+ \partial_r(J\partial_r \phi)\right).
\end{equation*}
In the local coordinates $(s,n)$, where $n=r/\eps$, the Laplacian can be written as
\begin{equation*}
\Delta =(1-\eps n\kappa)^{-1}\left( \eps^{-2}\partial_n
  (1-\eps n\kappa)\partial_n +\partial_s
  \big((1-\eps n\kappa)^{-1}\partial_s\big)\right).
\end{equation*}
From this we readily deduce the  representation
\begin{equation}\label{LaplaceExpansion}
\Delta= \eps^{-2}\partial^2_n-\eps^{-1}\kappa\partial_n
-n\kappa^2\partial_n+\partial^2_s+\eps (2n\kappa\partial^2_s+n\kappa'\partial_s-n^2\kappa^3\partial_n)+\eps^2 P_\eps,
\end{equation}
where $P_\eps$ is a PDE of the second order on $s$ and the first one on $n$ whose coefficients  are uniformly bounded in $\omega$ with respect to $\eps$.
Then using this representation and \eqref{AsymptExps} we obtain the equation $-\partial^2_n w=\lambda qw$ in $\omega$.
Next, the eigenfunction $u_\eps$ as an element of $\dom T_\eps$ satisfies the conditions
\begin{equation}\label{PECoupling}
  [u_\eps]_{-\eps}=0, \quad [u_\eps]_{\eps}=0, \quad [\partial_r u_\eps]_{-\eps}=0, \quad [\partial_r u_\eps]_{\eps}=0,
\end{equation}
where  $[\,\cdot\,]_t$ stands for  the jump of a function across $\gamma_t$. These conditions  imply $v^\pm=w(\,\cdot\,,\pm1)$ and $\partial_nw(\,\cdot\,,\pm1)=0$, where $v^\pm$ denote the one-side traces of $v$ on $\gamma$, i.e., $v^\pm=v|_{\gamma_\pm}$. Combining the latter equalities, we can formally deduce that the pair $(v,w)$ must be an eigenvector of  the spectral problem
\begin{align}\label{P0EqV}
        -&\Delta v+av=\lambda \rho v \quad\text{in }\Omega\setminus\gamma,\qquad \ell v=0\quad\text{on } \partial\Omega,\\\label{P0EqW}
         -&\partial^2_n w=\lambda q w   \:\text{ in } \omega,\quad
        \partial_n w(s,-1)=0,\;\: \partial_n w(s,1)=0,\\\label{P0Coupling}
         &v^-=w(s,-1),\quad v^+=w(s,1),\quad s\in S
\end{align}
with the spectral parameter $\lambda$; \eqref{P0EqV}-\eqref{P0Coupling} will be regarded as the \textit{limit problem}.

We will show that the  problem can be associated with a non-self-adjoint matrix operator. Moreover this operator is not similar to  a self-adjoint one, because it possesses generalized eigenvectors. It means that the limit problem admits no self-adjoint realization. It it worth mentioning that the similar non-self-adjoint problem appeared in \cite{GomezNazarovPerez2006}, where the spectral problem for a membrane surrounded by a thin stiff band was studied. Actually  perturbed  spectral problem \eqref{PEproblem} and the problem in \cite{GomezNazarovPerez2006} give us nontrivial examples of self-adjoint operators acting on varying Hilbert spaces  that converge in some sense to a non-self-adjoint operator.

\section{Properties of the limit operator}\label{Sect3}
We use the following notation. The spectrum and resolvent set of a linear operator $T$ are
denoted by $\sigma(T)$ and $\varrho(T)$, respectively.  Let $T^*$ denote   the adjoint operator of $T$.  For any  $z\in \varrho(T)$, $\rRz(T)=(T-zI)^{-1}$ is the resolvent operator. Here and subsequently, $I$ is an identity operator.

\subsection{Spectrum of the limit operator}
We introduce the  operators
\begin{align*}
&\mathring{A}=\rho ^{-1}(-\Delta+a)\quad \text{in } L_2(\rho,\Omega), \quad
\dom\mathring{A}=\{f\in W_2^2(\Omega\setminus\gamma)\colon \ell f=0\;\; \text{on }\partial\Omega\},\\
&
B=-q^{-1}\partial^2_n \quad \text{in }L_2(q,\omega), \quad
\dom B=\left\{g\in W_2^{2,0}(\omega)\colon \partial_n g(\,\cdot\,,-1)=\partial_n g(\,\cdot\,,1)=0\right\},
\end{align*}
where $W_2^{2,0}(\omega)$ is the anisotropic Sobolev space
\begin{equation*}
  W_2^{2,0}(\omega)=\left\{g\in L_2(\omega)\colon \partial^k_n g\in L_2(\omega) \text{ for } k=1,2\right\}.
\end{equation*}
In the space $\cL=L_2(\rho,\Omega)\times L_2(q,\omega)$  we consider the matrix operator
\begin{multline*}
    \cP=
    \begin{pmatrix}
        \mathring{A} & 0 \\
        0 & B
    \end{pmatrix},\quad
    \dom\cP=\bigl\{(f, g)\in \dom\mathring{A}\times \dom B\colon\\
               f^-=g(\,\cdot\,,-1),\:\; f^+=g(\,\cdot\,,1)\bigr\}.
\end{multline*}
Now \eqref{P0EqV}--\eqref{P0Coupling}  can be written as $\cP u=\lambda u$ with the notation $u=(v,w)^\top$. For simplicity of notation we often write $(v,w)$ instead of $(v,w)^\top$.
Note that the operator $\cP$ is non-self-adjoint. Direct computations show that
\begin{multline*}
    \cP^*=
    \begin{pmatrix}
        A & 0\\
        0 & \mathring{B}
    \end{pmatrix},\quad
    \dom \cP^*=\bigl\{(f, g)\in \dom A\times W_2^{2,0}(\omega)\colon \\
               \partial_r f^-=-\partial_n g(\,\cdot\,,-1),\:\; \partial_r f^+=\partial_n g(\,\cdot\,,1)\bigr\},
\end{multline*}
where $A$ is the restrictions of $\mathring{A}$ to $\dom A=\{f\in \dom\mathring{A}\colon  f=0\; \text{ on } \gamma\}$
and $\mathring{B}$ is the extension of $B$ to the whole space $W_2^{2,0}(\omega)$. Here $ \partial_r f^\pm$ stand for  the one-side traces of the normal derivative of $f$ on $\gamma$.

\begin{lem}\label{LemSpectrumB}
 The spectrum of  $B$ consists of a countable set of real eigenvalues of  infinite multiplicity. Moreover, $\lambda$ belongs to $\sigma(B)$ if and only if  $\lambda$ is an eigenvalue of the Sturm-Liouville problem
  \begin{equation}\label{SturmLiouville}
   y''+\lambda q(n) y=0,\;\; n\in(-1,1), \quad y'(-1)=0,\;\; y'(1)=0.
 \end{equation}
\end{lem}
\begin{proof}
Obviously, the operator $B$ is  self-adjoint.
For given $\lambda\in \mathbb{C}$ and $g\in L_2(q,\omega)$, the  equation $(B-\lambda I)\phi=g$ can be treated as the boundary value problem for the ordinary differential equation
\begin{equation*}
         -(\partial^2_n+\lambda q(n)) \phi= q(n)g(s,n)
      \; \text{ in } \omega,\quad
        \partial_n\phi(s,-1)=0,\;\:\partial_n\phi(s,1)=0
\end{equation*}
with the parameter $s\in S$.  If $\lambda$ is not an eigenvalue of \eqref{SturmLiouville},  then the problem has a unique solution $\phi(s,\,\cdot\,)$ for each $g(s,\,\cdot\,)\in L_2(-1,1)$ and almost all $s\in S$. Moreover $\phi$ belongs to $L_2(q,\omega)$  in view of its integral representation via the Green function.
Otherwise, if $\lambda$ is  an eigenvalue of \eqref{SturmLiouville} with eigenfunction $y$,   the problem  is unsolvable for some right-hand sides $g$, because the corresponding homogeneous problem has infinitely many linearly independent  solutions of the form $b(s)y(n)$, where $b\in L_2(\gamma)$.
Therefore the spectrum of $B$ consists of all eigenvalues  of  \eqref{SturmLiouville}, and  the corresponding eigenspaces  are infinite-dimen\-sional.
\end{proof}

\begin{thm}
The operator $\cP$ has real discrete  spectrum. Moreover $$\sigma(\cP)=\sigma(A)\cup\sigma(B).$$
\end{thm}
\begin{proof}
Let us construct the resolvent of $\cP$ in an explicit form. Given $\mu\in \mathbb{C}$, $f\in L_2(\rho,\Omega)$ and $g\in L_2(q,\omega)$, we write $(\mathring{A}-\mu I)v=f$,  $(B-\mu I)w=g$.
The second equation  admits a unique solution $w=\rR(B)g$ if $\mu\in \varrho(B)$, and then
 $v$ is a solution of the problem
\begin{equation}\label{NonhomogEqU1}
\begin{aligned}
  -&\Delta v+a v-\mu \rho v=\rho f \quad\text{in }\Omega\setminus\gamma,\quad \ell v=0\quad\text{on } \partial\Omega,\\
  &v^-=w(\,\cdot\,,-1), \quad v^+=w(\,\cdot\,,1).
\end{aligned}
\end{equation}
Suppose the  operator
$T(\mu) \colon W_2^2(\omega)\to L_2(\rho,\Omega)$  solves the problem
\begin{equation}\label{HomogenPrU1}
\begin{aligned}
   -&\Delta \phi+a\phi-\mu \rho \phi=0 \quad\text{in }\Omega\setminus\gamma,\quad \ell \phi=0\text{ on } \partial\Omega,\\
  &\phi^-=\psi(\,\cdot\,,-1), \quad \phi^+=\psi(\,\cdot\,,1)
\end{aligned}
\end{equation}
for a given function $\psi\in W_2^2(\omega)$. If $\mu\in \varrho(A)$, then the problem has a unique solution $\phi=T(\mu)\psi$ and hence $T(\mu)$ is bounded.
Next, $v$ can be represented as the sum of a solution of  \eqref{NonhomogEqU1} subject to the homogeneous  boundary  conditions on $\gamma_\pm$ and a solution of \eqref{HomogenPrU1} with $\psi=w$.
Therefore
\begin{equation*}
v=\rR(A)f+T(\mu)w=\rR(A)f+T(\mu)\rR(B)g,
\end{equation*}
provided $\mu\in \varrho(A)\cap \varrho(B)$.
Then the resolvent of $\cP$ can be written in the form
\begin{equation*}
    \rR(\cP)=
    \begin{pmatrix}
        \rR(A) & T(\mu)\rR(B)\\
              0        & \rR(B)
    \end{pmatrix}.
\end{equation*}
The equality $\sigma(\cP)=\sigma(A)\cup\sigma(B)$ follows directly from this representation and the fact that $T(\mu)$ is bounded for $\mu\in \varrho(A)$. Obviously, $\sigma(A)\cup\sigma(B)\subset\sigma(\cP)$.  Suppose, contrary to our claim, that  $\sigma(\cP)$ is not contained in $\sigma(A)\cup\sigma(B)$. Then there exists $\mu\in \sigma(\cP)$ such that $\mu\in \rho(A)\cap\rho(B)$. Hence, the operators $\rR(A)$,  $\rR(B)$ and $T(\mu)$ are bounded, and therefore, the operator $\rR(\cP)$ is also bounded, a contradiction. Clearly,  $\sigma(\cP)$ is real and discrete, because $A$ and $B$ are self-adjoint operators with discrete spectra.
\end{proof}

\begin{cor}
 The operator $\cP$ has non-compact resolvent.
\end{cor}
\begin{proof}
  The entry  $\rR(B)$ of the matrix $\rR(\cP)$ is a non-compact operator, which is clear from Lemma~\ref{LemSpectrumB}. The main reason for which $\rR(B)$ is not compact is that the second derivative $\partial^2_n$ is a non-elliptic PDE in the two-dimensional domain $\omega$.
\end{proof}

\begin{rem}
The operator $A$  can be represent as the direct sum   $A_-\oplus A_+$, where
\begin{align*}
  &A_- = \rho^{-1}(-\Delta+a),\quad  \dom A_-=\{f\in W_2^2(\Omega_-)\colon \ell f=0\; \text{ on }\partial\Omega,\; f=0\; \text{ on } \gamma\},\\
  &A_+ = \rho^{-1}(-\Delta+a),\quad  \dom A_+=\{f\in W_2^2(\Omega_+)\colon  f=0 \text{ on } \gamma\}.
\end{align*}
Hence
$\sigma(\cP)=\sigma(A_-)\cup\sigma(A_+)\cup\sigma(B)$.
\end{rem}

\subsection{Structure of  generalized eigenspaces}

Let $X_\mu$ be the generalized eigenspa\-ce corresponding to $\mu\in \sigma(\cP)$, i.e., $ X_\mu=\{h\in\cL\colon h\in
\ker(\cP-\mu I)^k \text{ for some } k\in\mathbb{N}\}$. If a vector $h$ belongs to $\ker(\cP-\mu I)^j$ and $(\cP-\mu I)^{j-1}h\neq 0$, one says that $h$ is a generalized eigenvector of rank $j$. The eigenspace $X_\mu^{0}=\ker(\cP-\mu I)$ is a subspace of $X_\mu$; eigenvectors are precisely the generalized eigenvectors of rank $1$.

\subsubsection{Case $\lambda\in \sigma(A)\setminus\sigma(B)$.}
We look first for non-trivial solutions $u=(v,w)$ of  $\cP u=\lambda u$.
Since $\lambda$ does not belong to $\sigma(B)$,  we have  $w=0$.
This in turn implies $v|_\gamma=0$, by \eqref{P0Coupling}. Therefore $v$ must be a eigenvector of $A$ corresponding to $\lambda$. All eigenvalues of $A$ have finite multiplicities.  Assume that $\lambda$ is an eigenvalue of  multiplicity $K$ and  $V_1$,\dots,$V_K$ are the eigenfunctions of $A$ such that
\begin{equation}\label{NormalizeCondVj}
  \int_{\Omega}\rho V_iV_j\,dx=\delta_{i,j},\qquad i,j=1,\ldots,K.
\end{equation}
Here  $\delta_{i,j}$ is the Kronecker symbol. Then $X_\lambda^{0}$ is spanned by
\begin{equation}\label{BasisCaseI}
         u_1=(V_1,0),\;\;u_2=(V_2,0),\;\;\dots\;\;,u_K=(V_K,0).
\end{equation}
The generalized eigenvectors of rank $2$ satisfy  the equation
$(\cP-\lambda I)u_*=u$, where $u$ is an eigenvector of $\cP$. There is no such vectors $u_*=(v_*, w_*)$ in this case, because the second components of all the eigenvectors $u_k$ are zero. In fact,  $B w_*=\lambda w_*$ implies  $w_*=0$. Therefore $(A-\lambda I)v_*=v$, where $v$ is a linear combination of $V_j$. The last equation is unsolvable, since $A$ is self-adjoint.

\subsubsection{Case $\lambda\in \sigma(B)\setminus\sigma(A)$.}
Suppose that $y$ is an eigenfunction of \eqref{SturmLiouville} corresponding to  $\lambda$.   Then $B$ has a countable collection of linearly independent eigenfunctions
$\big\{b_j(s) y(n)\big\}_{j=1}^\infty$, where $\{b_j\}_{j=1}^\infty$ is a basis in $L_2(\gamma)$ consisting of smooth functions. On the other hand, the problem
\begin{equation}\label{ProblemVz}
        -\Delta v+av=\lambda \rho  v \:\text{ in }\Omega\setminus\gamma,\quad \ell v=0\: \text{ on } \partial\Omega,
         \quad v^-=y(-1)b ,\quad v^+= y(1)b
\end{equation}
is uniquely solvable for any $b\in C^\infty(\gamma)$, since $\lambda$ does not belong to $\sigma(A)$. Therefore
$\cP$ possesses the countable set of linearly independent eigenvectors
\begin{equation}\label{BasisCaseII}
u_1=(v_1, b_1 y),\; u_2=(v_2, b_2 y),\;\dots\;,u_j=(v_j, b_j y),\;\dots,
\end{equation}
where $v_j$ is a solution of \eqref{ProblemVz} with $b_j$ in place of $b$. Note that the values $y(\pm1)$  are different from zero and hence all $v_j$ are  non-trivial solutions.
In this case, finding generalized eigenvectors of rank $2$  leads to the  equation $(B-\lambda I)w_*=by$, which is unsolvable for $b\neq 0$.  Hence $X_\lambda= X^0_\lambda$ and $\dim X_\lambda=\infty$.

\subsubsection{Case $\lambda\in \sigma(A)\cap\sigma(B)$.}
Since $\lambda$ belongs to $\sigma(B)$,  any eigenvector of $\cP$ has the form $(v(x), b(s) y(n))$, where $v$ is a solution of \eqref{ProblemVz}. But now we can not solve problem \eqref{ProblemVz} for any function $b$, since $\lambda$ is a point of $\sigma(A)$.

\begin{prop}\label{PropSolvabilityV*}
Let $f\in L_2(\Omega)$ and $b_\pm\in W_2^{3/2}(\gamma)$.
  Assume that $\lambda$ is an eigenvalue of $A$ of  multiplicity $K$ and the corresponding  eigenfunctions $V_1,\dots,V_K$ satisfy \eqref{NormalizeCondVj}. Then the problem
  \begin{equation}\label{ProblemVbF}
        -\Delta v+av=\lambda \rho  v +f\:\text{ in }\Omega\setminus\gamma,\quad \ell v=0\: \text{ on } \partial\Omega,
         \quad v^-=b_-,\quad v^+=b_+
\end{equation}
admits a solution $v\in W_2^2(\Omega\setminus \gamma)$ if and only if
\begin{equation}\label{SolvabilityVbF}
  \int_\gamma \big( b_+\partial_r V_k^+ - b_-\partial_r V_k^-\big)\,d\gamma+\int_\Omega fV_k\,dx=0
\end{equation}
for all $k=1,\dots,K$.
\end{prop}
\begin{proof}
This statement is a simple consequence of  the Fredholm alternative for the self-adjoint operator $A$ with compact resolvent.
Conditions \eqref{SolvabilityVbF} can be  obtained
by multiplying the equation in \eqref{ProblemVbF} by $V_1,\dots,V_K$ in turn and then integrating by parts
twice in view of the boundary conditions.
\end{proof}

In view of Proposition~\ref{PropSolvabilityV*}, problem \eqref{ProblemVz} is solvable if and only if $(b,\Psi_k)_{L_2(\gamma)}=0$
for all $k=1,\dots,K$, where
\begin{equation}\label{PsiKset}
\Psi_k=y(1)\,\partial_r V_k^+-y(-1)\,\partial_r V_k^-.
\end{equation}
Let $H_\lambda$ be the subspace in $L_2(\gamma)$ spanned by $\Psi_1,\dots,\Psi_K$. Hence the solvability of \eqref{ProblemVz} is equivalent to the orthogonality of $b$ to $H_\lambda$.

\begin{prop}
Assume $\lambda$ is an eigenvalue of $A=A_-\oplus A_+$  and $k_\pm$ is the multiplicity of $\lambda$ in the spectrum  of $A_\pm$. Then $\dim H_\lambda\geq \max(k_+,k_-)$.
\end{prop}
\begin{proof}
  Since the operator $A$ is a direct sum of $A_+$ and $A_-$, we can choose a basis in the  eigenspace of $A$ such that $V_1,\dots,V_{k_+}$ are identically equal to zero in $\Omega_-$ and the rest of $k_-$ eigenfunctions $V_{k_++1},\dots,V_{K}$ are identically equal to zero in $\Omega_+$. Note that $k_++k_-=K$.
Then $H_\lambda$ is the linear span of
\begin{equation}\label{BaseInHlambda}
\partial_r V_1^+,\,\ldots\,, \partial_r V_{k_+}^+,\,  \partial_r V_{k_++1}^-,\,\ldots\,,  \partial_r V_{K}^-.
\end{equation}
In general, these normal derivatives are  linearly dependent in $L_2(\gamma)$, but the first of $k_+$ derivatives are always linearly independent as well as the last of $k_-$ ones.
Suppose, contrary to our claim, that the functions $\partial_r V_1^+,\,\ldots\,, \partial_r V_{k_+}^+$ are  linearly dependent in $L_2(\gamma)$. Then there exists an eigenfunction $v$ of the Dirichlet type problem $-\Delta v+av=\lambda \rho  v$ in $\Omega_+$, $v=0$ on $\partial\Omega_+$ such that $\partial_r v=0$ on the whole boundary of $\Omega_+$, but this is impossible. The same conclusion can be drawn for the second part of the normal derivatives.
\end{proof}

Assume $\dim H_\lambda=d$ and choose a basis $\{b_j\}_{j=1}^\infty$ in $L_2(\gamma)$ such that $b_1$,\dots,$b_d$ belong to $H_\lambda$, while $b_k$ for $j>d$ are elements of $H^\perp_\lambda$, and $b_j\in C^\infty(\gamma)$.
Then
\begin{equation}\label{BasisCaseIII}
  (V_1,0),\,\dots,\,(V_K,0),\,(v_{d+1}, b_{d+1}y),\,(v_{d+2}, b_{d+2}y),\dots
\end{equation}
is a countable set of linearly independent eigenvectors of  $\cP$.
Here $v_j$ is a solution of \eqref{ProblemVz} for $b=b_j$ which is orthogonal to the span of $V_1,\dots,V_K$ in $L_2(\rho, \Omega)$.  So $X_\lambda^{0}$ is infinite-dimensional.

In this case, we can also find the generalized eigenvectors $u_*=(v_*, w_*)$ of rank~$2$. They  satisfy  the equation
$(\cP-\lambda)u_*=u$,  where $u=(v, w)$ is an eigenvector of $\cP$. Hence $v_*$ and $w_*$  solve the problem
\begin{align}\label{RootEqV*}
-&\Delta v_*+a v_*=\lambda \rho v_*+\rho v \quad\text{in }\Omega\setminus\gamma,\quad \ell v_*=0\quad\text{on } \partial\Omega,\\\label{RootEqW*}
-&\partial^2_n w_*=\lambda qw_*+qw  \quad\text{in } \omega,
         \qquad\partial_n w_*(\,\cdot\,,-1)=0,\quad\partial_n w_*(\,\cdot\,,1)=0,
\\\label{RootCoupling}
&v_*^-=w_*(\,\cdot\,,-1),\quad v_*^+=w_*(\,\cdot\,,1).
\end{align}
It is easily seen that  \eqref{RootEqV*}--\eqref{RootCoupling} is unsolvable if $w\neq 0$.
Therefore the ge\-ne\-ra\-li\-zed eigenvectors can be  associated  with some non-trivial linear combinations of the  eigenvectors $(V_1,0),\dots,(V_K,0)$ only.
We set $w=0$ and $v=c_1V_1+\cdots+c_KV_K$. Then $w_*(s,n)=b_*(s)y(n)$ and
\begin{align}\label{RootEqVStar}
       -&\Delta v_*+av_*=\lambda \rho v_*+\rho\sum_{j=1}^K c_j V_j\quad\text{in }\Omega\setminus\gamma,\quad \ell v_*=0\quad\text{on } \partial\Omega,\\\label{RootCouplingVStar}
        &v_*^-=y(-1)\,b_*,\quad v_*^+=y(1)\,b_*.
\end{align}
By Proposition~\ref{PropSolvabilityV*} and  \eqref{NormalizeCondVj}, the  problem admits a solution if and only if
\begin{equation*}
c_1=-(b_*,\Psi_1)_{L_2(\gamma)},\,\dots\,, c_K=-(b_*,\Psi_K)_{L_2(\gamma)}.
\end{equation*}
Since $v$ is non-zero, $b_*$ must have a non-trivial projection onto the subspace $H_\lambda$. But $b_1$,\dots,$b_d$ have this property.
Putting $b_*=b_j$ for $j=1,\dots,d$ into \eqref{RootEqVStar}, \eqref{RootCouplingVStar}, we obtain the Jordan chains of length $2$
\begin{equation*}
  (V^{(j)},0)\;\;\mapsto\;\; (v^{(j)}_{*},b_j y),\quad j=1,\ldots,d,
\end{equation*}
where $V^{(j)}=c_1^{(j)}V_1+\cdots+c_K^{(j)}V_j$ with  $c_i^{(j)}=-(b_j,\Psi_i)_{L_2(\gamma)}$.

There are no ge\-ne\-ra\-li\-zed eigenvectors of  rank  $3$, because all the eigenvectors of  rank  $2$ have  nonzero second components and the equation $(B-\lambda I)y_*=b_jy$ is unsolvable for $b_j\neq 0$. Hence the space $X_\lambda$  has the basis consisting of:
\begin{itemize}
  \item[$\circ$] $d$ Jordan chains of length $2$
  \begin{equation}\label{Xlambdabasis1}
    (V^{(1)},0)\longmapsto (v^{(1)}_{*},b_1 y  ),\,\ldots\,, (V^{(d)},0)\longmapsto (v^{(d)}_{*},b_d y  ),
  \end{equation}
  \item[$\circ$]  the series of eigenvectors of the form
  \begin{equation}\label{Xlambdabasis2}
  (v_{d+1}, b_{d+1} y ),\, (v_{d+2}, b_{d+2} y ),\,\dots\,,(v_j, b_j y ),\,\dots\,,
  \end{equation}
  \item[$\circ$]  $K-d$ eigenvectors of the form
  \begin{equation}\label{Xlambdabasis3}
  (V^{(d+1)},0),\,\dots\,,(V^{(K)},0).
  \end{equation}
\end{itemize}
Here $V^{(1)},\dots,V^{(K)}$ are linearly independent eigenfunctions of $A$ corresponding to eigenvalue $\lambda$.

We summarize the information about the spectrum and generalized eigenspaces of the limit operator $\cP$ that we have obtained.

\begin{thm}\label{TheoremSpectrumStructure}
Let $X_\lambda$ be the generalized eigenspace corresponding to $\lambda\in \sigma(\cP)$ and $X_\lambda^{0}$ be the eigenspace.

 \textit{(i)} If $\lambda\in \sigma(A)\setminus \sigma(B)$, then $X_\lambda$ is finite dimensional,  $X_\lambda=X_\lambda^{0}$, and the  basis in $X_\lambda$ is given by \eqref{BasisCaseI}.

\textit{(ii)} The part $\sigma(B)\setminus \sigma(A)$ of $\sigma(\cP)$ consists of eigenvalues $\lambda$ of infinite multiplicity with eigenspaces $X_\lambda=X_\lambda^{0}$ generated by vectors \eqref{BasisCaseII}.

 \textit{(iii)} The part $\sigma(A)\cap\sigma(B)$ is also consists of eigenvalues $\lambda$ of infinite multiplicity, but $X_\lambda\neq X_\lambda^{0}$. Apart from the eigenvectors,
      there exist the generalized  eigenvectors of rank $2$, and the basis in $X_\lambda$ is given by \eqref{Xlambdabasis1}--\eqref{Xlambdabasis3}. The dimension of  factor space $X_\lambda/X_\lambda^{0}$
      does not exceed the multiplicity of $\lambda$ in the spectrum of $A$, namely
      \begin{equation*}
       \max(k_-,k_+)\leq \dim X_\lambda/X_\lambda^{0}\leq k_-+k_+,
      \end{equation*}
 where $k_-$ and $k_+$ are the multiplicity of $\lambda$ in the spectra  of $A_-$ and $A_+$ respectively.
\end{thm}

\begin{rem}\label{remZeroP}
  The operator $\cP$ has always the eigenvalue $\lambda=0$ of  infinite multiplicity, since $0\in \sigma(B)$.
  The zero eigenvalue is the smallest infinite-fold one, because $B$ is non-negative. All negative eigenvalues, if they exist, have finite multiplicities.
\end{rem}

\section{Asymptotics of eigenvalues  in the case $\lambda_0\in \sigma(B)\setminus \sigma(A)$ }\label{SectCaseII}

We will first focus our attention on perturbations of infinite-fold eigenvalues. In this section, we  construct the asymptotics of countable set of eigenvalues $\lambda_\eps$ of \eqref{PEproblem} that converge to an infinite-fold eigenvalue $\lambda_0$ of  $\cP$ when $\lambda_0\not\in \sigma(A)$.  We look for the asymptotics in the form
\begin{equation}\label{AsympExUI}
\lambda^\eps\sim\lambda_0+\lambda_1\eps+\cdots,\quad u_\eps(x)\sim
    \begin{cases}
        v_0(x)+\eps v_1(x)+\cdots & \text{if } x\in \Omega\setminus \omega_\eps,\\
        w_0(s,\rep)+\eps w_1(s,\rep)+\cdots & \text{if } x\in \omega_\eps,
    \end{cases}
\end{equation}
where $(v_0,w_0)$ is an non-zero element of $X_{\lambda_0}$.
To match the expansions on $\gamma_{\pm\eps}$, we write  $v_k$ in the local coordinates $(s,n)$. Then \eqref{PECoupling} becomes
\begin{gather*}
  w_0(s,\pm1)+\eps w_1(s,\pm1)-
    v_0(s,\pm\eps)-\eps v_1(s,\pm\eps)+\cdots\sim 0,
\\
  \eps^{-1}\partial_n w_0(s,\pm1)+ \partial_n w_1(s,\pm1)-\partial_r v_0(s,\pm\eps)-\eps \partial_r v_1(s,\pm\eps)+\cdots\sim 0.
\end{gather*}
After expanding   $v_k(s,\pm\eps)$ and their derivatives into the Taylor series about $n=\pm 0$ for fixed $s$,  we in particular derive $ w_1(\,\cdot\,,\pm1)=v_1^\pm\pm\partial_r v_0^\pm$ and $\partial_n w_1(\,\cdot\,,\pm1)=\partial_r v_0^\pm$.
Of course, $w_0(\,\cdot\,,\pm1)=v_0^\pm$ and   $\partial_n w_0(\,\cdot\,,\pm1)=0$, since $(v_0,w_0)\in \dom \cP$.

Substituting \eqref{AsympExUI} into  \eqref{PEproblem}, and taking into account  representation \eqref{LaplaceExpansion}, we obtain that the pair $(v_1,w_1)$  solves the problem
\begin{align}\label{EqV1}
        &-\Delta v_1+av_1=\lambda_0 \rho v_1+\lambda_1 \rho v_0\quad\text{in }\Omega\setminus\gamma,\qquad \ell v_1=0\quad\text{on } \partial\Omega,\\\label{EqW1new}
         &-\partial^2_n w_1=\lambda_0 qw_1-\kappa \partial_nw_0+\lambda_1 qw_0
         \quad\text{in } \omega,\\\label{NeumannCondW1}
        &\phantom{-}\partial_n w_1(\,\cdot\,,-1)=\partial_r v_0^-,\quad
        \partial_n w_1(\,\cdot\,,1)=\partial_r v_0^+,\\\label{CouplingV1W1}
         &\phantom{-}v_1^-=w_1(\,\cdot\,,-1)+\partial_r v_0^-,\quad
          v_1^+=w_1(\,\cdot\,,1)-\partial_r v_0^+.
\end{align}
The compatibility condition for \eqref{EqW1new}, \eqref{NeumannCondW1} reads
\begin{equation}\label{CompactibilityW1}
y(-1)\,\partial_r v_0^--y(1)\,\partial_r v_0^+
+   \kappa\int_{-1}^1\partial_n w_0\, y\,dn=   \lambda_1 \int_{-1}^1q w_0y\,dn\quad\text{on }\gamma,
\end{equation}
where $y$ is an eigenfunction of \eqref{SturmLiouville} corresponding to $\lambda_0$.
We have applied the Fredholm alternative for the operator with compact resolvent associated with the Sturm-Liouville problem \eqref{SturmLiouville}. The relation \eqref{CompactibilityW1} can be treated as a spectral equation on $\lambda_1$ for a  pseudodifferential  operator on~$\gamma$.
We introduce the Dirichlet-to-Neumann maps $N^\pm_\lambda$ in $L_2(\gamma)$ as follows. Let $z$ be the solution of  problem
 \begin{equation}\label{DtoNProblemMinus}
   -\Delta z+az=\lambda \rho z\;\;\text{in }\Omega_-,\quad z=\varphi\quad\text{on }\gamma,\;\;\ell z=0\quad\text{on }\partial\Omega
 \end{equation}
for given $\varphi$.  Set $N^-_\lambda\varphi=\partial_r z|_{\gamma}$. We follow \cite{ArendtMazzeo} in assuming that
\begin{equation*}
\dom N^-_\lambda=\big\{\varphi\in L_2(\gamma)\colon z\in W_2^1(\Omega_-) \text{ and } \partial_r z|_{\gamma}\in L_2(\gamma)\big\}.
\end{equation*}
Likewise, we set $N^+_\lambda\varphi=-\partial_r z|_{\gamma}$, where $z$ is a solution of the problem
\begin{equation}\label{DtoNProblemPlus}
  -\Delta z+az=\lambda \rho z\;\;\text{in }\Omega_+,\quad z=\varphi\quad\text{on }\gamma,
\end{equation}
and $\dom N^+_\lambda=\big\{\varphi\in L_2(\gamma)\colon z\in W_2^1(\Omega_+) \text{ and } \partial_r z|_{\gamma}\in L_2(\gamma)\big\}$.
The operators $N^\pm_\lambda$ transform the Dirichlet data on $\gamma$ for solutions of the corresponding boundary value problems into the Neumann ones. Both operators are  well-defined if $\lambda\not\in \sigma(A)$.
The minus sing in definition of $N^+_\lambda$ indicates that the direction of axis $r$  coincides with the inward normal on
$\gamma=\partial \Omega_+$.

It follows  from \eqref{BasisCaseII} that $w_0(s,n)=b_0(s)y(n)$  and $v_0$ is a solution of \eqref{ProblemVz} for $b=b_0$.
Then $\partial_r v_0^-=y(-1)\,N^-_{\lambda_0} b_0$ and $\partial_r v_0^+=-y(1)\,N^+_{\lambda_0} b_0$. Consequently, condition \eqref{CompactibilityW1} reads
\begin{equation*}
y^2(-1)\,N^-_{\lambda_0} b_0+y^2(1)\,N^+_{\lambda_0} b_0
+   \kappa b_0\int_{-1}^1y y'\,dn=   \lambda_1 b_0 \int_{-1}^1q y^2\,dn.
\end{equation*}
Suppose  $\|y\|_{L_2(q,(-1,1) )}=1$
and write $\theta_{\lambda_0}^\pm=y^2(\pm1)$.
Since
\begin{equation*}
 \int_{-1}^1y y'\,dn=\tfrac12\int_{-1}^1 (y^2)'\, dn=\tfrac12\left(\theta_{\lambda_0}^{+}-\theta_{\lambda_0}^{-}\right),
\end{equation*}
we can finally rewrite \eqref{CompactibilityW1}  in the form $ N_{\lambda_0} b_0=\lambda_1 b_0$,
where
\begin{equation}\label{OperatorNlambda}
N_\lambda= \theta_\lambda^{-}N^-_\lambda+\theta_\lambda^{+}N^+_\lambda +\tfrac12\,(\theta_\lambda^{+}-\theta_\lambda^{-})\kappa.
\end{equation}
Note that the values $\theta_{\lambda_0}^-$ and $\theta_{\lambda_0}^+$ do not depend on our choice of  $y$, because all eigenvalues of  the Sturm-Liouville problem are simple.

\begin{prop}
  The operator  $N_\lambda$ is self-adjoint, bounded below and has compact resolvent for all $\lambda\not\in \sigma(A)$.
\end{prop}
\begin{proof}
The operators $N^-_\lambda$ and $N^+_\lambda$ are self-adjoint in $L_2(\gamma)$, bounded below and have compact
  resolvents \cite[Th.3.1]{ArendtMazzeo}. The linear combination $\theta_\lambda^{-}N^-_\lambda+\theta_\lambda^{+}N^+_\lambda$ has the same properties, since $\theta_\lambda^{-}$ and $\theta_\lambda^{+}$ are positive.
  Finally, $N_\lambda$ is a perturbation of this linear combination by the operator of multiplication by the bounded function $\tfrac12\,(\theta_\lambda^{+}-\theta_\lambda^{-})\kappa(s)$, which completes the proof.
\end{proof}

Denote by $\{\lambda_1^{(k)}\}_{k=1}^\infty$ the eigenvalues of $N_{\lambda_0}$. So we have calculated  the countable set of correctors $\lambda_1^{(k)}$ in asymptotics \eqref{AsympExUI}. To keep the mathematics rather simple we  suppose that the spectrum of $N_{\lambda_0}$ is simple. It means that the infinite-fold eigenvalue $\lambda_0$ of $\cP$
asymptotically splits into an infinite number of simple eigenvalues of $T_\eps$ under perturbation.
Let $\{b^{(k)}\}_{k=1}^\infty$ be the collection of orthonormal eigenfunctions of $N_{\lambda_0}$.

Let us fix the corrector $\lambda_1^{(k)}$ and  set $w_0=w^{(k)}=b^{(k)}y$. Then $v_0=v^{(k)}$ is a unique solution of \eqref{ProblemVz} with $\lambda_0$ and $b^{(k)}$ in place of $\lambda$ and $b$ respectively.
Such a choice of $(v^{(k)}, w^{(k)})$ ensures that problem \eqref{EqW1new}, \eqref{NeumannCondW1} is solvable, since the compatibility condition
$N_{\lambda_0} b^{(k)}=\lambda_1^{(k)} b^{(k)}$ holds. The problem admits the family of solutions
\begin{equation*}
w_1(s,n)=\hat{w}_1^{(k)}(s,n)+b_1(s)y(n),
\end{equation*}
where $b_1$ is an arbitrary $L_2(\gamma)$-function and $\hat{w}_1^{(k)}$ is the partial solution  subject to the condition
\begin{equation}\label{OrthoCondW1}
\int_{-1}^1 qy \hat{w}_1^{(k)}\,dn=0 \quad\text{ for all } s\in S.
\end{equation}
According to \eqref{EqV1} and \eqref{CouplingV1W1}, with this choice of $w_1$ the function $v_1$ can be written as $v_1=\hat{v}_1^{(k)}+\tilde{v}_1$,
where $\hat{v}_1^{(k)}$ (resp. $\tilde{v}_1$) is a solution of the problem
\begin{equation}\label{U1problemI}
        -\Delta \phi+a \phi=\lambda \rho\phi+\lambda_1^{(k)} \rho v^{(k)}\text{ in }\Omega\setminus\gamma,\quad \ell \phi=0\text{ on } \partial\Omega,
         \quad\phi^-=g_-,\;\; \phi^+=g_+
\end{equation}
with the Dirichlet data $g_\pm=\hat{w}_1^{(k)}(\,\cdot\,,\pm 1)\mp \partial_n v^{(k)}|_{\gamma_\pm}$
(resp. $g_\pm=y(\pm 1) b_1$).
The term $\hat{v}_1^{(k)}$ is uniquely defined, while $\tilde{v}_1$ along with $b_1$ will be fixed below.

Set $a_0(s)=a(s,0)$. The pair $(v_2, w_2)$ in the asymptotics of $u_\eps$  solves the problem
\begin{align}
\label{EqV2}
  &-\Delta v_2+av_2=\lambda_0 \rho v_2+\lambda_1 \rho v_1+\lambda_2 \rho v_0\quad\text{in }\Omega\setminus\gamma,\quad \ell v_2=0\quad \text{on }\partial\Omega,
\\\label{EqW2}
  & -\partial^2_n w_2=\lambda_0 qw_2-(\kappa \partial_n-\lambda_1 q)w_1-
        (\kappa^2 n \partial_n-\partial_s^2+a_0-\lambda_2 q)w_0
         \quad\text{in } \omega,
\\\label{NeumannCondW2}
  &\phantom{-}\partial_n w_2(\,\cdot\,,-1)=\partial_r v_1^--\partial^2_r v_0^-,\quad \partial_n w_2(\,\cdot\,,1)=\partial_r v_1^++\partial^2_r v_0^+,
\\\label{CouplingV2W2}
  &\phantom{-}v_2^-=w_2(\,\cdot\,,-1)+\partial_r v_1^--\partial^2_r v_0^-, \quad v_2^+=w_2(\,\cdot\,,1)-\partial_r v_1^+-\partial^2_r v_0^+.
\end{align}
By reasoning as above, we obtain that the solution $w_2$ exists if and only if
\begin{equation}\label{Nz1-lz1=f1}
  (N_{\lambda_0}-\lambda_1^{(k)}I) b_1=\lambda_2 b^{(k)}+h_1^{(k)},
\end{equation}
where
\begin{multline}\label{H1}
  h_1^{(k)}=
  y(1)\big(\partial_r \hat{v}_1^{(k)}|_{\gamma_+}+\partial^2_r v^{(k)}|_{\gamma_+}\big)-y(-1)
  \big(\partial_r \hat{v}_1^{(k)}|_{\gamma_-}-\partial^2_r v^{(k)}|_{\gamma_-}\big)\\
 -  \int_{-1}^1(\kappa \partial_n-\lambda_1^{(k)} q) y \hat{w}_1^{(k)}\,dn -
  \int_{-1}^1 \big(\kappa^2n y y'\, b^{(k)}-y^2(\partial_s^2b^{(k)}-a_0b^{(k)}) \big)\,dn.
\end{multline}
Since $\lambda_1^{(k)}$ is a simple eigenvalue of $N_{\lambda_0}$, the second corrector $\lambda_2=\lambda_2^{(k)}$ can be uniquely calculated from the solvability condition for equation \eqref{Nz1-lz1=f1}
\begin{equation*}
  \lambda_2^{(k)}=-(h_1^{(k)},b^{(k)})_{L_2(\gamma)}.
\end{equation*}
Moreover, there exists a unique solution $b_1^{(k)}$ of \eqref{Nz1-lz1=f1} satisfying the condition $(b_1^{(k)},b^{(k)})_{L_2(\gamma)}=0$.
Hence $v_1^{(k)}$ is now uniquely defined by choosing $\tilde{v}_1$ to be a solution of \eqref{U1problemI} with the Dirichlet data $g_\pm=y(\pm1) b_1^{(k)}$.
The compatibility condition \eqref{Nz1-lz1=f1} allows us to solve problems \eqref{EqW2}, \eqref{NeumannCondW2} and \eqref{EqV2}, \eqref{CouplingV2W2} one after another and to find solutions $w_2^{(k)}$ and $v_2^{(k)}$.

The process used to ﬁnd the leading terms of asymptotic expansions \eqref{AsympExUI} can be continued to systematically construct all other terms.
For fixed $\lambda_0\in \sigma(B)\setminus \sigma(A)$ we consider the countable collection of formal approximations to eigenvalues and eigenfunctions of the perturbed problem
\begin{gather}\label{hatLambdaEps}
\Lambda_\eps^{(k)} =\lambda_0+\lambda_1^{(k)}\eps+\lambda_2^{(k)}\eps^2+\lambda_3^{(k)}\eps^3,\\\label{hatUeps}
\hat{u}_\eps^{(k)} (x)=
    \begin{cases}
        v_0^{(k)}(x)+\eps v_1^{(k)}(x)+\eps^2 v_2^{(k)}(x)+\eps^3 v_3^{(k)}(x) & \text{if } x\in \Omega\setminus \omega_\eps,\\
w_0^{(k)}(s,\rep)+\eps w_1^{(k)}(s,\rep)+\eps^2 w_2^{(k)}(s,\rep)+\eps^3 w_3^{(k)}(s,\rep) & \text{if } x\in \omega_\eps.
    \end{cases}
\end{gather}

Let $H$ be a Hilbert space with  norm $\|\cdot\|$ and let $T$ be a self-adjoint operator in $H$ with a domain $\dom T$.
We say a pair $(\mu, u)\in \Real\times \dom T$ is a \textit{quasimode} of  $T$ with the accuracy $\delta$, if $\|u\|=1$  and
$\|(T-\mu I)u\|\leq\delta$. Of course, if $\delta=0,$ then $\mu$ is an eigenvalue of $T$ with the eigenvector $u$.

\begin{prop}[\hglue-0.1pt{\cite[p.139]{PDEVinitiSpringer}}]\label{LemQuasimodes}
   Assume $(\mu, u)$ is a quasimode of $T$ with accuracy $\delta>0$ and  the spectrum of $T$ is discrete in  the interval
$[\mu-\delta, \mu+\delta]$. Then there exists an eigenvalue $\mu_*$ of  $T$ such that $|\mu_*-\mu|\leq\delta$.
\end{prop}

In order to construct the quasimodes of $T_\eps$, we must modify the approximations of eigenfunctions, because
they do not  belong to $\dom T_\eps$. By construction, the functions $v_j^{(k)}$ and $w_j^{(k)}$ are smooth, since the coefficients $a$, $\rho$ and $q$ are smooth.
But, in general,  $\hat{u}_\eps^{(k)}$ have jump discontinuities on  $\partial\omega_\eps$.

\begin{figure}[h]
  \centering
  \includegraphics[scale=2]{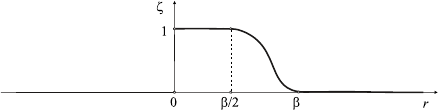}\\
  \caption{Plot of the function $\zeta$.}\label{FigPlotZeta}
\end{figure}

Let us define the function $\zeta$ plotted in Fig.~\ref{FigPlotZeta}.  This function is smooth outside the origin,  $\zeta(r)=1$ for $r\in [0,\beta/2]$ and $\zeta(r)=0$ in the set $\Real\setminus [0,\beta)$.
We  can choose $\beta$ small enough such that the local coordinates $(s,r)$ are well defined in $\omega_{2\beta}$. Set
\begin{multline}\label{EtaCorrector}
  \eta_\eps^{(k)}=\big([\hat{u}_\eps^{(k)}]_{\eps}+[\partial_r\hat{u}_\eps^{(k)}]_{\eps}\,
  (r-\eps)\big)\,\zeta(r-\eps)\\
  +\big([\hat{u}_\eps^{(k)}]_{-\eps}+[\partial_r\hat{u}_\eps^{(k)}]_{-\eps}\,(r+\eps)\big)
  \,\zeta(-r-\eps).
\end{multline}
The function is different from zero in the set $\omega_{\beta+\eps}\setminus\omega_\eps$ only. And it is easy to check  that $\eta_\eps^{(k)}$ and $\partial_r\eta_\eps^{(k)}$ have the same jumps across the boundary of $\omega_\eps$ as $\hat{u}_\eps^{(k)}$ and $\partial_r\hat{u}_\eps^{(k)}$ respectively.
Therefore the function $U_\eps^{(k)}=\hat{u}_\eps^{(k)}-\eta_\eps^{(k)}$
belongs to the domain of $T_\eps$. Moreover we have not changed $\hat{u}_\eps^{(k)}$ too much, since
\begin{equation}\label{EtaEpsEstimate}
  \sup_{x\in \Omega\setminus \overline{\omega}_\eps}\big(|\eta_\eps^{(k)}(x)|+|\Delta\eta_\eps^{(k)}(x)|\big)\leq c\eps^3.
\end{equation}
It follows from the explicit formula for $\eta_\eps^{(k)}$ and  smallness of
 jumps of $\hat{u}_\eps^{(k)}$, $\partial_r\hat{u}_\eps^{(k)}$ across $\partial\omega_\eps$. All the jumps are of order $O(\eps^3)$, as $\eps\to 0$, by construction.

We will henceforth write $n_\eps^{(k)}=\|U_\eps^{(k)}\|^{-1}_{L^2(\rho_\eps,\Omega)}$.

\begin{lem}\label{LemmaQuasimodesTeps}
 For each $k\in \mathbf{N}$, the pair $(\Lambda_\eps^{(k)}, n_\eps^{(k)}U_\eps^{(k)})$ constructed above is a quasimode of the operator $T_\eps$ with the accuracy $O(\eps^2)$ as $\eps\to 0$.
\end{lem}

\begin{proof}
For simplicity we shall drop the index $k$ in the sequel and write $(\Lambda_\eps, U_\eps)$, $v_j$, $w_j$ and $\lambda_j$ instead of $(\Lambda_\eps^{(k)}, U_\eps^{(k)})$, $v_j^{(k)}$, $w_j^{(k)}$ and $\lambda_j^{(k)}$. Set $F_\eps=(T_\eps-\Lambda_\eps)U_\eps$. Then we have
\begin{equation*}
 F_\eps=(-\Delta+a-\Lambda_\eps \rho)U_\eps=\sum_{j=0}^3\eps^j(-\Delta v_j+av_j-\rho\sum_{i=0}^{j}\lambda_i  v_{j-i})
     +\Delta\eta_\eps-a\eta_\eps+\Lambda_\eps \rho\eta_\eps
\end{equation*}
outside $\omega_\eps$. From our choice of $v_j$, we derive that the first sum in the right-hand side vanishes.
Therefore $\sup_{x\in\Omega\setminus \omega_\eps}|F_\eps(x)|\leq c_1\eps^3$, because of \eqref{EtaEpsEstimate}.
Applying representation \eqref{LaplaceExpansion} of the Laplace operator in $\omega$, we have
\begin{equation*}
\Delta -a+\eps^{-2}\Lambda_\eps q
= \sum_{j=0}^3\eps^{j-2} p_j+\eps^2 P_\eps-a_\eps,
\end{equation*}
where $p_0=\partial^2_n+\lambda_0 q$, $p_1=-\kappa \partial_n+\lambda_1 q$, $p_2=-n\kappa^2 \partial_n+\partial_s^2-a(s,0)+\lambda_2 q$,
\begin{equation*}
  p_3=2n\kappa\partial^2_s+n\kappa'\partial_s-n^2\kappa^3\partial_n-na(s,0)\partial_n+\lambda_3 q,
\end{equation*}
and $a_\eps(s,n)=a(s,\eps n)-a(s,0)-\eps n\partial_na(s,0)$.
Then, in the  domain $\omega_\eps$, we obtain
\begin{equation*}
  F_\eps=(-\Delta+a-\eps^{-2}\Lambda_\eps q)U_\eps=\sum_{j=0}^3\eps^{j-2} \sum_{i=0}^{j-i}p_iw_{j-i}-\eps^2 P_\eps U_\eps+a_\eps=-\eps^2 P_\eps U_\eps+a_\eps,
\end{equation*}
since the functions $w_0,\ldots,w_3$ solve the equations $\sum_{i=0}^{j-i}p_iw_{j-i}=0$, by construction.
Consequently $\sup_{x\in\omega_\eps}|F_\eps(x)|=\eps^2 \sup_{x\in\omega_\eps}|P_\eps U_\eps(x)|+\sup_{x\in\omega_\eps}|a_\eps(x)|\leq c_2\eps^2$.
Hence
\begin{equation*}
  \|F_\eps\|^2_{L^2(\rho_\eps,\Omega)}\leq \int_{\Omega\setminus \omega_\eps}\rho|F_\eps|^2\,dx+\eps^{-2}\int_{\omega_\eps}q|F_\eps|^2\,dx\leq c_3\eps^6+c_4|\omega_\eps|\eps^2\leq c_5\eps^3.
\end{equation*}

The main contribution  to the $L^2(\rho_\eps,\Omega)$-norm of $U_\eps$ is given by the integral $\eps^{-2}\int_{\omega_\eps}q|w_0|^2\,dx$ which is of order $O(\eps^{-1})$, as $\eps\to 0$. Hence $\|U_\eps\|_{L^2(\rho_\eps,\Omega)}\geq c_6\eps^{-1/2}$ for $\eps$ small enough, i.e., $n_\eps\leq c_7\eps^{1/2}$.
This bound along with
$ \|F_\eps\|_{L^2(\rho_\eps,\Omega)}\leq c_8\eps^{3/2}$ yields
\begin{equation}\label{EstQuasimodes}
 \|(T_\eps-\Lambda_\eps I)(n_\eps U_\eps)\|_{L^2(\rho_\eps,\Omega)}=n_\eps\|F_\eps\|_{L^2(\rho_\eps,\Omega)}\leq c_9 \eps^{2},
\end{equation}
and this is precisely the assertion of the lemma.
\end{proof}

\begin{thm}\label{ThmCaseII}
Suppose that $\lambda_0$ is an eigenvalue of the limit operator $\cP$ such that $\lambda_0\in\sigma(B)\setminus \sigma(A)$.
Assume the spectrum $\{\lambda_1^{(k)}\}_{k\in \mathbb{N}}$ of the  operator $N_{\lambda_0}$ is simple.
Then there exists a countable set of eigenvalues $\lambda_{\eps,k}$, $k\in \mathbb{N}$, in the spectrum of $T_\eps$ that
converge to $\lambda_0$  and admit the asymptotics
\begin{equation}\label{lambdaEpsAsymp}
\lambda_{\eps,k}=\lambda_0+\eps\lambda_1^{(k)}+O(\eps^2), \quad\text{as }  \eps\to 0.
\end{equation}
\end{thm}

\begin{proof}
Fix $I\in \mathbb{N}$.
In view of  Proposition~\ref{LemQuasimodes} and Lemma~\ref{LemmaQuasimodesTeps}, there exist  eigenva\-lues $\lambda_{\eps,1},\dots,\lambda_{\eps,I}$ of $T_\eps$ and a constant $c_I$ such that
\begin{equation}\label{LambdaEpsIEsts}
  |\lambda_{\eps,k}-\lambda_0-\eps\lambda_1^{(k)}|\leq c_I\eps^2
\end{equation}
for $k=1,2,\dots,I$ and $\eps$ small enough. Moreover these eigenvalues are pairwise different. Suppose, contrary to our claim, that some eigenvalue $\lambda_\eps$ of $T_\eps$ simultaneously satisfies two  estimates \eqref{LambdaEpsIEsts}, for example when $k=1$ and $k=2$. Then $\lambda_\eps-\lambda_0-\eps\lambda_1^{(1)}\leq c_I\eps^2$ and $\lambda_0+\eps\lambda_1^{(2)}-\lambda_\eps\leq c_I\eps^2$. Adding these inequalities yields $\lambda_1^{(2)}-\lambda_1^{(1)}\leq 2c_I\eps$
for all $\eps$ small enough. But this is impossible, because the spectrum of  $N_{\lambda_0}$ is simple and therefore $\lambda_1^{(1)}<\lambda_1^{(2)}$. Write
\begin{equation*}
  \delta_I^\eps=\eps \max_{k=1,\dots,I}|\lambda_1^{(k)}|+c_I\eps^2.
\end{equation*}
Hence the interval $[\lambda_0-\delta_I^\eps,\lambda_0+\delta_I^\eps]$ contains at least $I$ eigenvalues of  $T_\eps$ that possess the asymptotics \eqref{lambdaEpsAsymp}.
\end{proof}

\begin{figure}[t]
  \centering
  \includegraphics[scale=0.9]{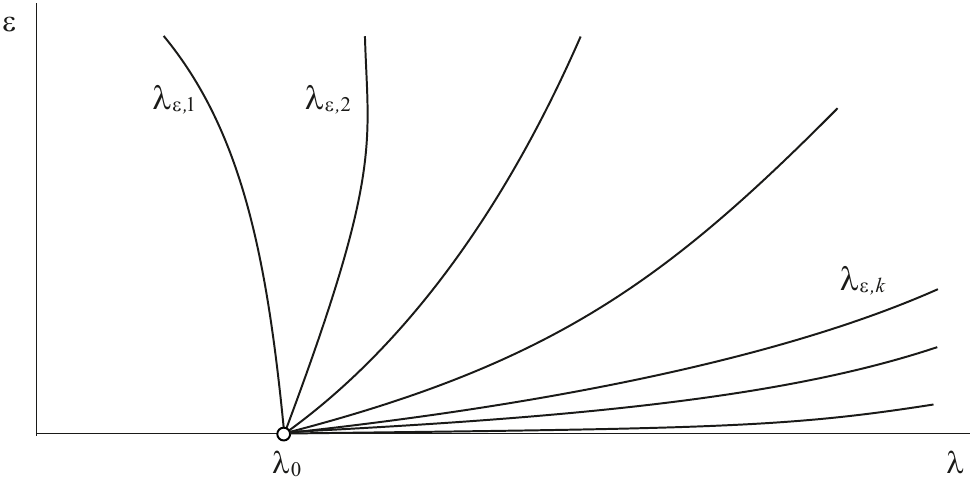}\\
  \caption{Bifurcation of the eigenvalue $\lambda_0\in \sigma(B)\setminus \sigma(A)$.}\label{FigBifurcation}
\end{figure}

\section{Asymptotics of eigenvalues  in the case $\lambda_0\in \sigma(A)\cap\sigma(B)$} \label{SectCaseIII}
In view of Theorem~\ref{TheoremSpectrumStructure}, if the set $\sigma(A)\cap\sigma(B)$ is non-empty, the  operator $\cP$ possesses generalized  eigenvectors of rank $2$. This requires changing the structure of  asymptotics
\begin{align}\label{AsympExLambdaIII}
\lambda^\eps&\sim\lambda_0+\lambda_{1/2}\eps^{1/2}+\lambda_1\eps+\cdots,\\\label{AsympExUIII}
u_\eps(x)&\sim
    \begin{cases}
        v_0(x)+\eps^{1/2}v_{1/2}(x)+\eps v_1(x)+\cdots & \text{for } x\in \Omega\setminus \omega_\eps,\\
        w_0(s,\rep)+\eps^{1/2}w_{1/2}(s,\rep)+\eps w_1(s,\rep)+\cdots & \text{for } x\in \omega_\eps.
    \end{cases}
\end{align}
Here $(v_0,w_0)$ is an non-zero element of the eigenspace spanned by vectors \eqref{BasisCaseIII}. If $y$ is a normalized eigenfunction of \eqref{SturmLiouville} corresponding to  $\lambda_0$, then $w_0(s,n)=b_0(s)y(n)$ for some $b_0\in L_2(\gamma)$.
As above, substituting the series in \eqref{PEproblem} in particular yields
\begin{align}\label{EqV1/2}
        &-\Delta v_{1/2}+a v_{1/2}=\lambda_0 \rho v_{1/2}+\lambda_{1/2} \rho v_0 \text{ in }\Omega\setminus\gamma,\quad \ell v_{1/2}=0\text{ on } \partial\Omega,\\\label{EqW1/2}
         & -\partial^2_s w_{1/2}=\lambda_0 qw_{1/2}+\lambda_{1/2} qw_0
         \quad\text{in } \omega,\\\label{NeumannCondW1/2}
        &\phantom{-}\partial_n w_{1/2}(\,\cdot\,,-1)=0,\quad
         \partial_n w_{1/2}(\,\cdot\,,1)=0,\\\label{CouplingV1/2W1/2}
         &\phantom{-}v_{1/2}^-=w_{1/2}(\,\cdot\,,-1),\quad
          v_{1/2}^+=w_{1/2}(\,\cdot\,,1).
\end{align}
Since \eqref{EqW1/2}, \eqref{NeumannCondW1/2} can be written as $(B-\lambda_0)w_{1/2}=\lambda_{1/2}b_0y$ and $B$ is self-adjoint, a solution $w_{1/2}$  exists  if only if
\begin{equation}\label{lambda12b0}
\lambda_{1/2} b_0=0.
\end{equation}
 This condition is a branching point of our algorithm.

\subsection*{Integer power asymptotics: case $w_0\neq 0$}\label{SecIPA}
If $b_0$ is a non-zero function, then $\lambda_{1/2}=0$ by necessity. Then problem \eqref{EqV1/2}--\eqref{CouplingV1/2W1/2} turns into
the limit problem \eqref{P0EqV}--\eqref{P0Coupling} and $(v_{1/2},w_{1/2}) \in X^0_{\lambda_0}$.
Without loss of generality we assume that  $(v_{1/2},w_{1/2})=0$, i.e.,  this vector  is absorbed by the leading term  of the asymptotics. Moreover, a trivial verification shows that all terms $\lambda_{j/2}$, $(v_{j/2},w_{j/2})$ with half-integer indexes in \eqref{AsympExLambdaIII}, \eqref{AsympExUIII}
can be treated as equal to zero. In this case,  we come back to the integer power asymptotics  \eqref{AsympExUI}, but the construction of quasimodes needs a slight modification.
The next terms $\lambda_1$, $(v_1, w_1)$ solve  problem \eqref{EqV1}--\eqref{CouplingV1W1}, and therefore  condition \eqref{CompactibilityW1} must hold. But now we cannot rewrite \eqref{CompactibilityW1} in the form of the spectral equation for $N_{\lambda_0}$, because this operator  is not defined for $\lambda_0\in \sigma(A)$.

We will ``extend'' $N_\lambda$ to $\sigma(A)$  by means of a restriction of the space in which it acts. A slight change in the proof of Proposition~\ref{PropSolvabilityV*} actually shows that both problems \eqref{DtoNProblemMinus} and \eqref{DtoNProblemPlus} are solvable for $\lambda \in \sigma(A)$ if the function $\varphi$ in the boundary condition on $\gamma$ is orthogonal to the subspace $H_\lambda\subset L_2(\gamma)$ spanned by functions \eqref{PsiKset}. Although  solutions of the problems, in this case, are ambiguously determined, we can subject them to some additional condition. Namely, there exists a unique solution $z$ of \eqref{DtoNProblemMinus} (or \eqref{DtoNProblemPlus})  satisfying the condition $\partial_r z|_{\gamma}\perp H_\lambda$.
So we can define the Dirichlet-to-Neumann map   $M^-_\lambda\varphi=\partial_r z|_{\gamma}$ on $H_\lambda^\perp$, where $z$ is a solution of \eqref{DtoNProblemMinus} belonging to $H_\lambda^\perp$. Similarly, we define $ M^+_\lambda\varphi=-\partial_r z|_{\gamma}$,  where $z\in H_\lambda^\perp$ is a solution of \eqref{DtoNProblemPlus}.
Both operators are well defined for  $\lambda\in \mathbb{C}$. In fact, we have
\begin{equation*}
 M^\pm_\lambda=(I-P_\lambda)N^\pm_\lambda(I-P_\lambda),
\end{equation*}
 where $P_\lambda$ is the orthogonal projector onto the subspace $H_\lambda$.
Moreover, $M^\pm_\lambda=N^\pm_\lambda$ for  $\lambda\in\mathbb{C}\setminus\sigma(A)$, since the subspace $H_\lambda$ is trivial in this case.

Now solvability condition \eqref{CompactibilityW1} becomes $M_{\lambda_0}b_0=\lambda_1 b_0$, where
$$
M_\lambda= \theta_\lambda^- M^-_\lambda+\theta_\lambda^+ M^+_\lambda+\tfrac12(\theta_\lambda^{+}
-\theta_\lambda^{-})\kappa.
$$
The pseudodifferential operator $M_\lambda$ has the same properties as $N_\lambda$.
Suppose that the spectrum of $M_{\lambda_0}$ is simple and $\lambda_1^{(k)}$ is the $k$-th eigenvalue of $M_\lambda$ with the normalized eigenfunction $b^{(k)}$.
We set $w_0=b^{(k)}y$.

Assume $\lambda_0$ is an $K$-fold eigenvalue of $A$ with the eigenspace $\mathcal{V}_{\lambda_0}$. In order to shorten notation, we introduce the vector $\bV=(V_1,\dots,V_K)$, where the eigenfunctions $V_k$ are subject to condition \eqref{NormalizeCondVj}.
Then the leading term $v_0=v_0^{(k)}$ in \eqref{AsympExUIII} solves \eqref{ProblemVz} for $\lambda=\lambda_0$ and $b=b^{(k)}$ and has the form
\begin{equation*}
 v_0^{(k)}(x)=\hat{v}_0^{(k)}(x)+\alpha_0\cdot \bV(x),
\end{equation*}
where $\alpha_0$ is an arbitrary vector in $\Real^K$ and $\hat{v}_0^{(k)}$ is a partial solution of \eqref{ProblemVz} such that  $\hat{v}_0^{(k)}\perp\mathcal{V}_{\lambda_0}$.
The dot denotes the scalar product in $\Real^K$.
To determine $v_0^{(k)}$ uniquely, we should calculate  $\alpha_0$. Next, we have
\begin{equation}\label{w1kRepresII}
  w_1(s,n)=\hat{w}_1^{(k)}(s,n)+b_1(s)y(n),
\end{equation}
where $\hat{w}_1^{(k)}$ is a partial solution of \eqref{EqW1new}, \eqref{NeumannCondW1} subject to condition \eqref{OrthoCondW1}, and $b_1$ is an arbitrary $L_2(\gamma)$-function. Assume that  $b_1=g_0+g$, where $g_0\in H_{\lambda_0}$ and $g\in H_{\lambda_0}^\perp$.
Then problem \eqref{EqV1}, \eqref{CouplingV1W1} for $v_1$ admits  solutions
\begin{equation}\label{v1kRepresII}
  v_1=\phi+\psi+\alpha_1\cdot \bV,\quad \alpha_1\in \Real^K,
\end{equation}
where $\phi$ and $\psi$ solve  the problems
\begin{align*}
       &-\Delta \phi+a\phi=\lambda_0 \rho\phi\text{ \ in }\Omega\setminus\gamma,\quad \ell\phi=0\text{ \ on } \partial\Omega,\quad
         \phi=y(\pm1)\, g\text{ \ on }\gamma_\pm;
\\
         &\begin{aligned}
        &-\Delta \psi+a \psi=\lambda_0 \rho\psi+\lambda_1^{(k)} \rho \left(\hat{v}_0^{(k)}+\alpha_0\cdot \bV\right)\text{ \ in }\Omega\setminus\gamma,\quad \ell\psi=0\text{ \ on } \partial\Omega,\\
         &\phantom{-}\psi^\pm= \hat{w}^{(k)}_1(\,\cdot\,,\pm1)\mp \partial_r \hat{v}_0^{(k)}+y(\pm1)\, g_0
         \mp\alpha_0\cdot \partial_r\bV \text{ \ on }\gamma
        \end{aligned}
\end{align*}
respectively. Since $g$ is orthogonal to  $H_{\lambda_0}$, the first problem is solvable and admits a solution  belonging to  $\mathcal{V}_{\lambda_0}^\perp$. As for the second one,  its solvability conditions  can be written in the form $\big(C_{\lambda_0}-\lambda_1^{(k)}I\big)\alpha_0=f$,
where $C_{\lambda_0}$ is a matrix with the entries
$$
   c_{ij}=\int_{\gamma_+} \partial_r V_i\:\partial_r V_j\,d\gamma+\int_{\gamma_-} \partial_r V_i\:\partial_r V_j\,d\gamma,\quad i,j=1,\dots,K,
$$
and $f$ is a vector  with the components
 \begin{multline*}
  f_i=\lambda_1^{(k)} \int_{\Omega}\rho \hat{v}_0^{(k)} V_i\,dx+ \int_{\gamma_+}\big(\hat{w}_1^{(k)}-\partial_r \hat{v}_0^{(k)}+y(1)\, g_0\big)\,\partial_r V_i\,d\gamma\\
  -\int_{\gamma_-}\big(\hat{w}_1^{(k)}+\partial_r \hat{v}_0^{(k)}+y(-1)\, g_0\big)\,\partial_r V_i\,d\gamma,
  \quad i=1,\dots,K.
 \end{multline*}
If we suppose that $\lambda_1^{(k)}$ is not an eigenvalue of $C_{\lambda_0}$, then the solvability conditions for $\psi$ can be fulfilled for any $g\in H_{\lambda_0}$. It is enough to set
$\alpha_0=\big(C_{\lambda_0}-\lambda_1^{(k)}I\big)^{-1}f$. Then the problem has a solution $\phi\in \mathcal{V}_{\lambda_0}^\perp$.

Note that the vector $\alpha_0$  has not yet been defined, because $f$ depends on the unknown function $g_0$.
Using representations \eqref{w1kRepresII} and \eqref{v1kRepresII}  along with the fact that $g\in H_{\lambda_0}^\perp$, we can write the solvability condition for \eqref{EqW2}, \eqref{NeumannCondW2} in the form
\begin{equation}\label{Mz1-lz1=f1}
  (M_{\lambda_0} -\lambda_1^{(k)}I)g=\lambda_2^{(k)} b^{(k)}+\lambda_1^{(k)}g_0+h_1^{(k)},
\end{equation}
where $h_1^{(k)}$ is given by \eqref{H1}.
For the equation \eqref{Mz1-lz1=f1} to be meaningful, we need to ensure that the  right hand side is orthogonal to  $H_{\lambda_0}$. Obviously, $b^{(k)}$ belongs to $H_{\lambda_0}^\perp$. Next, if $\lambda_1^{(k)}$ is different from zero, there exists a unique vector $g_0\in H_{\lambda_0}$ such that  $\lambda_1^{(k)}g_0 +h_1^{(k)}\in  H_{\lambda_0}^\perp$.
With  $g_0$ in hand, we can uniquely defined $f$, $\alpha_0$, and finally the leading term $(v_0^{(k)},w_0^{(k)})$.
The solvability condition for  \eqref{Mz1-lz1=f1} has the form $\lambda_2^{(k)}=-(h_1^{(k)}+\lambda_1^{(k)}g_0,b^{(k)})_{L_2(\gamma)}$.
Then there exists a unique solution $g$ satisfying the condition $(g,b^{(k)})_{L_2(\gamma)}=0$. And finally,  we can uniquely define  $w_1^{(k)}$  and $v_1^{(k)}$  (up to the vector $\alpha_1$). This process can be continued to systematically construct all other terms $\lambda_j^{(k)}$, $v_j^{(k)}$ and $w_j^{(k)}$ in the approximation $\Lambda_\eps^{(k)}$, $\hat{u}_\eps^{(k)}$ of the form \eqref{hatLambdaEps}, \eqref{hatUeps}. As in the previous section, $\hat{u}_\eps^{(k)}$ can be improved to the element $U_\eps^{(k)}=\hat{u}_\eps^{(k)}-\eta_\eps^{(k)}$ from the domain of operator $T_\eps$, where  $\eta_\eps^{(k)}$ is given by \eqref{EtaEpsEstimate}.

Summarizing results of the above calculations, we obtain the following statement.

\begin{lem}\label{LemmaQuasimodesTepsCaseIIIa}
The pairs $(\Lambda_\eps^{(k)},\|U_\eps^{(k)}\|^{-1}_{L^2(\rho_\eps,\Omega)}U_\eps^{(k)})$, $k\in \mathbb{N}$, are quasimodes of  $T_\eps$ with the accuracy $O(\eps^2)$ as $\eps\to 0$ that approximate the part of spectrum lying in a vicinity of  $\lambda_0\in \sigma(A)\cap\sigma(B)$.
\end{lem}
The lemma can be proved similarly to Lemma~\ref{LemmaQuasimodesTeps}.

\subsection*{Integer power asymptotics: case $w_0=0$}\label{SecIPA0}
Suppose now that $\lambda_{1/2}$ and $b_0$ are equal to zero simultaneously. In this case,  we also construct the integer power asymptotics  \eqref{AsympExUI}. Since $w_0=0$, problem \eqref{EqV1}--\eqref{CouplingV1W1} becomes
\begin{align}\label{EqV100}
        &-\Delta v_1+av_1=\lambda_0 \rho v_1+\lambda_1 \rho v_0\quad\text{in }\Omega\setminus\gamma,\qquad \ell v_1=0\quad\text{on } \partial\Omega,\\\label{EqW1new00}
         &-\partial^2_n w_1=\lambda_0 qw_1
         \text{ \ in } \omega, \quad \partial_n w_1(\,\cdot\,,-1)=\partial_r v_0^-,\quad
        \partial_n w_1(\,\cdot\,,1)=\partial_r v_0^+,\\\label{CouplingV1W100}
         &\phantom{-}v_1^-=w_1(\,\cdot\,,-1)+\partial_r v_0^-,\quad
          v_1^+=w_1(\,\cdot\,,1)-\partial_r v_0^+.
\end{align}
Compatibility condition \eqref{CompactibilityW1} for problem \eqref{EqW1new00} takes the form
\begin{equation}\label{YpartialV0}
y(-1)\,\partial_r v_0^--y(1)\,\partial_r v_0^+= 0.
\end{equation}
If we set $v_0=\alpha \cdot\bV$, then the last condition can be written as $\alpha \cdot\Psi=0$,
where $\alpha \in \Real^K\setminus\{0\}$, $\Psi=(\Psi_1,\dots,\Psi_K)$ and $\Psi_k$ are given by \eqref{PsiKset}. Let $y_*$ be the solution of the Cauchy problem $-y_*''=\lambda_0 q y_*$ in $(-1,1)$,\; $y_*(-1)=0$, \mbox{$y_*'(-1)=1$.}
Problem \eqref{EqW1new00} admits the solution $w_1(s,n)= y_*(n)\,\alpha \cdot\partial_r \bV^-(s)+y(n)b_1(s)$, where $b_1$ is an arbitrary $L_2(\gamma)$-function. Obviously, $\partial_n w_1(\,\cdot\,,-1)=\partial_r v_0^-$. The Lagrange identity yields $(yy_*'-y'y_*)|_{-1}^1=0$, and hence, $y_*'(1)=y(-1)/y(1)$. Then condition \eqref{YpartialV0} becomes $\alpha \cdot \partial_r \bV^+= y_*'(1)\,\alpha \cdot \partial_r \bV^- $, and we have
\begin{equation*}
  \partial_n w_1(\,\cdot\,,1)=y_*'(1)\,\alpha \cdot\partial_r \bV^-= \alpha \cdot \partial_r \bV^+=\partial_r v_0^+.
\end{equation*}
Next, we can rewrite \eqref{EqV100} and \eqref{CouplingV1W100} in the form
\begin{align}\label{v1finite00}
        &-\Delta v_1+av_1=\lambda_0 \rho v_1+\lambda_1 \rho \,\alpha \cdot\bV\;\;\text{in }\Omega\setminus\gamma,\quad \ell v_1=0\;\;\text{on } \partial\Omega,\\\label{condv1finite00}
         &\quad v_1^-=\alpha \cdot\partial_r \bV^-+y(-1)b_1,\quad    v_1^+=(y_*(1)-y_*'(1))\,\alpha \cdot\partial_r \bV^++y(1)b_1.
\end{align}
To achieve  solvability of the problem one needs to choose proper vectors $\alpha $ along with the parameter $\lambda_1$. Multiplying the equation in \eqref{v1finite00} by $V_1,\dots,V_K$ in turn and integrating by parts
twice in view of the boundary conditions \eqref{condv1finite00} yield the equation $L\alpha =\lambda_1\alpha $, where the symmetric matrix $L=(l_{ij})_{i,j=1}^K$ has the entries
\begin{equation}\label{MatrixLw0}
  l_{ij}=\int_\gamma \big(
  (y_*(1)-y_*'(1))\,\partial_r V_i^+\partial_r V_j^+
  +\partial_r V_i^-\partial_r V_j^-
  \big)\,d\gamma.
\end{equation}
Note that $L$ does not depend on $b_1$, because of \eqref{YpartialV0}.

Hence, the number $\lambda_1$ and non-zero vector $\alpha $ satisfy two conditions $L\alpha =\lambda_1\alpha $ and $\alpha \cdot\Psi=0$. We must find the eigenvectors of $L$ which are ``orthogonal'' to $\Psi$. The dimension of the space $h_{\lambda_0}=\{\alpha\in \mathbb{C}^K\colon \alpha\cdot\Psi=0\}$ is $K-d$, since $\dim H_{\lambda_0}=d$.
Assume that $L$ has the simple eigenvalues $\nu_1,\dots,\nu_p$ with the eigenvectors $\alpha_1,\dots,\alpha_p$ belonging to $h_{\lambda_0}$. Obviously, $p\le K-d$.   For any pair $(\nu_k, \alpha_k)$, we can solve \eqref{v1finite00}, \eqref{condv1finite00} and find
\begin{equation*}
v_1^{(k)}=\hat{v}_1^{(k)}+\alpha_1^{(k)}\cdot\bV
\end{equation*}
up to the vector $\alpha_1^{(k)}\in\Real^K$. Also, it follows from the first condition in \eqref{condv1finite00} that
\begin{equation*}
  b_1^{(k)}(s)= y(-1)^{-1}\hat{v}_1^{(k)}(s,-1)-\alpha_k \cdot\partial_r \bV(s,-1),
\end{equation*}
and now the function $w_1^{(k)}=y_*\;\alpha_k \cdot\partial_r \bV^-+y\;b_1^{(k)}$ is uniquely defined.
We continue in this fashion obtaining $p$  quasimodes of  $T_\eps$ of the form
\begin{gather}\label{hatLambdaEps00}
\Lambda_\eps^{(k)} =\lambda_0+\nu_k\eps+\lambda_{2,k}\eps^2+\lambda_{3,k}\eps^3,\\\label{hatUeps00}
W_\eps^{(k)}(x)=
    \begin{cases}
       \alpha_k \cdot\bV(x)+\eps v_1^{(k)}(x)+\eps^2 v_2^{(k)}(x)+\eps^3 v_3^{(k)}(x)-\eta_\eps^{(k)}(x) & \text{if } x\in \Omega\setminus \omega_\eps,\\
\phantom{\alpha_k \cdot\bV(x)+\;}\eps w_1^{(k)}(s,\rep)+\eps^2 w_2^{(k)}(s,\rep)+\eps^3 w_3^{(k)}(s,\rep) & \text{if } x\in \omega_\eps.
    \end{cases}
\end{gather}

Now we summarize the results of our calculations. The next lemma can be proved similarly to Lemma~\ref{LemmaQuasimodesTeps}.

\begin{lem}\label{LemmaQuasimodesTepsCaseIII0}
The pairs $(\Lambda_\eps^{(k)},\|W_\eps^{(k)}\|^{-1}_{L^2(\rho_\eps,\Omega)}W_\eps^{(k)})$, $k=1,\dots,p$, are quasimodes of  $T_\eps$ with the accuracy $O(\eps^2)$ as $\eps\to 0$ that approximate a part of spectrum lying in a vicinity of  $\lambda_0\in \sigma(A)\cap\sigma(B)$.
\end{lem}

\subsection*{Half-integer power asymptotics}

The set of quasimodes in Lemma~\ref{LemmaQuasimodesTepsCaseIIIa}  does not approximate all eigenvalues of $T_\eps$ that converge to $\lambda_0$.
We  assume that $\lambda_{1/2}$ in \eqref{AsympExUIII} is different from zero, and then $b_0=0$, by \eqref{lambda12b0}. Recalling now \eqref{BasisCaseIII},
we have $v_0=\beta_0\cdot \bV$, where $\beta_0$ is an arbitrary vector in $\Real^K$ such that $\|\beta_0\|=1$. In this case, we will use some finite-dimensional operator instead of $M_{\lambda_0}$ to split the limit multiple eigenvalue $\lambda_0$.
Reasoning as above we deduce that the problem \eqref{EqV1/2}-\eqref{CouplingV1/2W1/2} has a solution of the form $w_{1/2}(s,n)=b_{1/2}(s)y(n)$, $v_{1/2}=\hat{v}_{1/2}+\beta_{1/2}\cdot\bV$
where $\beta_{1/2}\in \Real^K$ and $\hat{v}_{1/2}$ is a partial solution of the problem
\begin{align}\label{EqV1/2RV}
        &-\Delta v_{1/2}+av_{1/2}=\lambda_0 \rho v_{1/2}+\lambda_{1/2} \rho v_0 \text{ \ in }\Omega\setminus\gamma,\quad \ell v_{1/2}=0\text{ \ on } \partial\Omega,
         \\\label{CouplingV1/2W1/2RV}
         &\phantom{-}v_{1/2}^-=y(-1)\, b_{1/2},\qquad
          v_{1/2}^+=y(1)\, b_{1/2}
\end{align}
that is orthogonal to  $\mathcal{V}_{\lambda_0}$ in $L_2(\rho, \Omega)$.
By Proposition~\ref{PropSolvabilityV*}, the solvability conditions for  \eqref{EqV1/2RV}, \eqref{CouplingV1/2W1/2RV} can be written in the vector form
\begin{equation}\label{AlhpaJRV}
   \int_\gamma b_{1/2}\Psi\,d\gamma+\lambda_{1/2}\beta_0=0.
\end{equation}
For the next terms we have
\begin{align}\label{EqV1II}
        &-\Delta v_1+av_1=\rho(\lambda_0 v_1+\lambda_{1/2}  v_{1/2}+\lambda_1  v_0) \text{ \ in }\Omega\setminus\gamma,\quad \ell v_1=0\text{ \ on } \partial\Omega,\\
        \label{EqW1II}
         & -\partial^2_n w_1=\lambda_0 qw_1+\lambda_{1/2} qw_{1/2}
         \quad\text{in } \omega,
         \\\label{NeumannCondW1II}
          &\phantom{-}\partial_n w_1(\,\cdot\,,-1)=\partial_r v_0^-,\qquad
          \partial_n w_1(\,\cdot\,,1)=\partial_r v_0^+,
          \\\label{CouplingV1W1II}
         &\phantom{-}v_1^-=w_1(\,\cdot\,,-1)+\partial_r v_0^-,\qquad
          v_1^+=w_1(\,\cdot\,,1)-\partial_r v_0^+.
\end{align}
Problem \eqref{EqW1II}, \eqref{NeumannCondW1II} is solvable if $y(1)\,\partial_r v_0^+-y(-1)\,\partial_r v_0^-+\lambda_{1/2} b_{1/2}=0.$
Since $\partial_r v_0^\pm=\beta_0\cdot \partial_r \bV^\pm$, it can be written in the form
\begin{equation*}
  \beta_0\cdot\big(y(1)\,\partial_r \bV^+-y(-1)\,\partial_r \bV^-\big)+\lambda_{1/2} b_{1/2}=0.
\end{equation*}
Using notation \eqref{PsiKset}, we have
\begin{equation}\label{CompactibilityW1II}
  \beta_0\cdot\Psi+\lambda_{1/2} b_{1/2}=0.
\end{equation}
Multiplying this equality by $\Psi^\top$, integrating over $\gamma$ and recalling  \eqref{AlhpaJRV}, we finally discover
 $\bigl(G_{\lambda_0} -\lambda_{1/2}^2\bigr)\,\beta_0=0$,
where $G_{\lambda_0}$ is the Gram matrix of $\Psi_1$,\ldots,$\Psi_K$.
This matrix is semi-positive and its rank is equal to  the dimension of $H_{\lambda_0}$.

Suppose  $\omega^2$ is a positive simple eigenvalue of $G_{\lambda_0}$ with the eigenvector $\beta_0$.
So there exist two different correctors $\lambda_{1/2}=\omega$ and $\lambda_{1/2}=-\omega$
in  asymptotics \eqref{AsympExLambdaIII} with the same leading term $v_0=\beta_0\cdot \bV$
in  approximation \eqref{AsympExUIII}. First assume that $\lambda_{1/2}=\omega$.  From \eqref{CompactibilityW1II}, we have
$b_{1/2}=-\omega^{-1} \beta_0\cdot\Psi$.
Up to a function $b_1$, we can find $w_1(s,n)=\hat{w}_1(s,n)+b_1(s)y(n)$,
where $\hat{w}_1$ is a partial solution of \eqref{EqW1II}, \eqref{NeumannCondW1II} subject to  the condition $(\hat{w}_1,y)_{L_2(q,(-1,1) )}=0$.
Next, problem \eqref{EqV1II}, \eqref{CouplingV1W1II} is solvable if and only if
 $\int_\gamma b_{1}\Psi\,d\gamma+\omega\beta_{1/2}+\lambda_{1}\beta_{0}=h$,
where
\begin{equation*}
  h=\int_{\gamma_+}\big(\hat{w}_1-\partial_r v_0\big)\,\partial_r \bV_k\,\,d\gamma-\int_{\gamma_-}\big(\hat{w}_1+\partial_r v_0\big)\,\partial_r \bV\,d\gamma -\omega\int_{\Omega} \rho \hat{v}_{1/2}\bV_k\,dx.
\end{equation*}
Also, a solution of the problem
\begin{align*}
& -\partial^2_n w_{3/2}=q(\lambda_0w_{3/2}+\omega w_1+\lambda_{1} w_{1/2})
         \quad\text{in } \omega,\\
&\quad\partial_n w_{3/2}(\,\cdot\,,-1)=\partial_r v_{1/2}^-,\quad
         \partial_n w_{3/2}(\,\cdot\,,1)=\partial_r v_{1/2}^+.
\end{align*}
exists if and only if
\begin{equation}\label{CompactibilityW3/2II}
  y(-1)\,\partial_r v_{1/2}^--y(1)\,\partial_r v_{1/2}^+=
  \omega b_{1}+\lambda_{1} b_{1/2}.
\end{equation}
Reasoning as in the previous step, we can rewrite this condition in the form
\begin{equation*}
  \bigl(G_{\lambda_0}-\omega^2 I\bigr)\beta_{1/2}=2\omega\lambda_{1}\beta_0-\omega h+f,
\end{equation*}
where
\begin{equation*}
f=y(-1)\, \int_{\gamma_-}\partial_r \hat{v}_{1/2} \Psi\,d\gamma  -y(1)\,\int_{\gamma_+}\partial_r \hat{v}_{1/2} \Psi\,d\gamma.
\end{equation*}
Since $\omega^2$ is an eigenvalue of $G_{\lambda_0}$, the  system admits a solution $\beta_{1/2}$ if and only if
\begin{equation}\label{Lambda1III}
\lambda_1=\tfrac1{2\omega}\,\beta_0\cdot (\omega h-f).
\end{equation}
Although the unit vector $\beta_0$ is defined up to the change of sign, $\lambda_1$ is uniquely defined by \eqref{Lambda1III}, because the transformation $\beta_0\mapsto -\beta_0$ implies that the vectors $h$ and $f$ also change their sings.
We fix this solution such that  $\beta_{1/2}\cdot\beta_0=0$. Then
\begin{equation*}
   b_{1}=\tfrac1{\omega}\,\Big(y(-1)\,\partial_r \hat{v}_{1/2}^--y(1)\,\partial_r \hat{v}_{1/2}^+-\beta_{1/2}\cdot\Psi-\lambda_{1} b_{1/2}\Big),
\end{equation*}
by \eqref{CompactibilityW3/2II}. Assuming that $\lambda_{1/2}^+=\omega$  we have calculated $\lambda_1^+$, $w_{1/2}^{(+)}$, $w_1^{(+)}$ and $v_{1/2}^{(+)}$ in  \eqref{AsympExLambdaIII}, \eqref{AsympExUIII}. We can continue in this fashion obtaining the next terms of the asymptotics.
Taking $\lambda_{1/2}^-=-\omega$  we can compute analogously the terms $\lambda_k^-$, $w_{k}^{(-)}$ and $v_{k}^{(-)}$ in the asymptotics of other eigenvalue and eigenfunction. A simple analysis of the foregoing formulas shows that $w_{1/2}^{(-)}=-w_{1/2}^{(+)}$ and $v_{1/2}^{(-)}=-v_{1/2}^{(+)}$. Moreover, by construction, we have
\begin{equation*}
  w_{1/2}^{(\pm)}(s,n)=\mp \omega^{-1}\,y(n)\, \beta_0\cdot\Psi(s), \qquad v_{1/2}^{(\pm)}(x)=\pm v_*(x),
\end{equation*}
where $v_*$ is a solution of the problem
\begin{equation*}
  -\Delta v+av=\lambda_0 \rho v+\omega \rho\, \beta_0\cdot\bV \text{ in }\Omega\setminus\gamma,\quad \ell v=0\text{ on } \partial\Omega,\quad
  v=-\omega^{-1} y(\pm1)\, \beta_0\cdot \Psi\text{ \ on } \gamma_\pm
\end{equation*}
that is orthogonal to  $\mathcal{V}_{\lambda_0}$.

Let us summarize the above considerations in the following lemma.
\begin{lem}\label{LemmaQuasimodesTepsCaseIIIb}
  Let $d$ be the dimension of $H_{\lambda_0}$, where $\lambda_0\in \sigma(A)\cap\sigma(B)$. Suppose that all non-zero eigenvalues $\omega_1^2,\dots, \omega_d^2$ of the matrix $G_{\lambda_0}$ are simple and $\beta_{0,1},\dots,\beta_{0,d}$ are the corresponding normalized eigenvectors. Then the operator $T_\eps$ possesses $d$ pairs of  the quasimodes $(\hat{\mu}_{\eps, j}^{\pm}, n^{\pm}_{\eps, j} V^{(\pm)}_{\eps, j})$, $j=1,\ldots,d$, with the accuracy $O(\eps^{3/2})$ as $\eps\to 0$, where $n^{\pm}_{\eps, j}$ is a normalizing factor and
  \begin{align*}
&\hat{\mu}_{\eps, j}^{\pm}=
\lambda_0\pm \eps^{1/2}\omega_j+\sum_{k=2}^6\eps^{k/2}\lambda_{k/2,j}^{\pm},\\
&V^{(\pm)}_{\eps, j}=
        \beta_{0,j}\cdot\bV\pm\eps^{1/2}v_{*,j} +\sum_{k=2}^6\eps^{k/2} v_{k/2,j}^{(\pm)}-\eta^{(\pm)}_{\eps, j}  \text{ \ in } \Omega\setminus \omega_\eps,\\
    &V^{(\pm)}_{\eps, j}=
       \mp \eps^{1/2}\omega_j^{-1}y(\tfrac{r}{\eps})\, \beta_{0,j}\cdot\Psi(s)+
        \sum_{k=2}^6\eps^{k/2}w_{k/2,j}^{(\pm)}(s,\tfrac{r}{\eps})  \text{ \ in }  \omega_\eps.
\end{align*}
The small correctors  $\eta^{(\pm)}_{\eps, j}$ are defined as in \eqref{EtaCorrector}  with $V^{(\pm)}_{\eps, j}$ in place of $\hat{u}_\eps^{(k)}$.
\end{lem}
\begin{proof}
 The proof differs from the proof of Lemma~\ref{LemmaQuasimodesTeps} only by estimate \eqref{EstQuasimodes}. The approximations to eigenvalues and eigenfunctions have been constructed up to order $O(\eps^3)$ and therefore the remainder term $F_\eps$ (with the notation of  Lemma~\ref{LemmaQuasimodesTeps}) can be also estimated by $\|F_\eps\|_{L^2(\rho_\eps,\Omega)}\leq c_1 \eps^{3/2}$.
But the leading term $w_0$ in this asymptotics is equal to zero and hence the $L^2(\rho_\eps,\Omega)$-norm of $V^{(\pm)}_{\eps, j}$ is bounded uniformly with respect to $\eps$. Then the normalizing factor $n^{\pm}_{\eps, j}$ tends to some positive number as $\eps\to 0$ and
 \begin{equation*}
 \|(T_\eps-\hat{\mu}_{\eps, j}^{\pm}I)(n^{\pm}_{\eps, j} V^{(\pm)}_{\eps, j})\|_{L^2(\rho_\eps,\Omega)}=n^{\pm}_{\eps, j}\|F_\eps\|_{L^2(\rho_\eps,\Omega)}\leq c_2 \eps^{3/2},
\end{equation*}
which is the desired conclusion.
\end{proof}

In view of  Lemmas~\ref{LemmaQuasimodesTepsCaseIIIa}--\ref{LemmaQuasimodesTepsCaseIIIb}, we can prove  the following result in the same way as Theorem~\ref{ThmCaseII}.

\begin{thm}\label{ThmCaseIII}
Suppose that the set $\sigma(A)\cap\sigma(B)$  is non empty and $\lambda_0$ is an eigenvalue of the limit operator $\cP$ belonging to this intersection.
\begin{itemize}
  \item[\textit{(i)}] Assume the spectrum $\{\lambda_1^{(k)}\}_{k\in \mathbb{N}}$ of the operator $M_{\lambda_0}$ is simple.
 Then there exists a countable set of eigenvalues $\lambda_{\eps,k}$, $k\in \mathbb{N}$, in the spectrum of $T_\eps$ that admit the asymptotics
\begin{equation*}
\lambda_{\eps,k}=\lambda_0+\eps\lambda_1^{(k)}+O(\eps^2), \quad\text{as }  \eps\to 0.
\end{equation*}
  \item[\textit{(ii)}] Assume $\dim H_{\lambda_0}=d$ and all positive eigenvalues $\omega_j^2$ of the  matrix $G_{\lambda_0}$ are simple. Then operator $T_\eps$  possesses $2d$ eigenvalues with the asymptotics
 \begin{align*}\label{MuEpsAsympIII}
&\mu_{\eps,j}^-=\lambda_0- \eps^{1/2}\omega_j+\eps\lambda_{1,j}^-+O(\eps^{3/2}), \\
&\mu_{\eps,j}^+=\lambda_0+\eps^{1/2}\omega_j+\eps\lambda_{1,j}^++O(\eps^{3/2}),\quad j=1,\dots,d,
\end{align*}
as $\eps\to 0$, where  $\lambda_{1,j}^\pm$ can be calculated from \eqref{Lambda1III}.
\item[\textit{(iii)}] Also, the operator $T_\eps$ can have no more than $K-d$ eigenvalues of possessing the asymptotics
\begin{equation*}
    \lambda^\eps_i=\lambda_0+\nu_i\eps+O(\eps^2),\quad \text{as } \eps\to 0,
\end{equation*}
 where $\nu_i$ are simple non-zero eigenvalues of the matrix  in \eqref{MatrixLw0}.
\end{itemize}
\end{thm}

In this case, the generalized eigenspace $X_{\lambda_0}$ contains the Jordan chains. Note that
problem \eqref{EqV1/2}--\eqref{CouplingV1/2W1/2} for $(v_{1/2},w_{1/2})$ coincides up to the multiplier $\lambda_{1/2}$ in the right-hand side with  problem \eqref{RootEqV*}--\eqref{RootCoupling} for generalized eigenvectors. Inspecting the structure of $V^{(\pm)}_{\eps, j}$ more closely we  see that
\begin{equation*}
  V^{(\pm)}_{\eps, j}=U_j\pm\eps^{1/2} U^*_j+O(\eps), \text{ as } \eps\to 0,
\end{equation*}
where the vectors $U_j=(\beta_{0,j}\cdot\bV(x),0)$, $U^*_j=\big(v_{*,j}(x), -\omega_j^{-1}y(n)\, \beta_{0,j}\cdot\Psi(s)\big)$ form a Jordan chain of $\cP$ corresponding to $\lambda_0$. This observation has the following  geometric interpretation. The operator $T_\eps$ possesses  pairs of eigenvalues with the asymptotics $\mu_{\eps,j}^\pm=\lambda_0\pm\omega_j\eps^{1/2}+O(\eps)$ for which the corresponding normalized eigenfunctions $u_{\eps, j}^{(-)}$ and $u_{\eps, j}^{(+)}$
converge in $L_2(\Omega)$ to the same function $v_{0,j}=\beta_{0,j}\cdot\bV$, as $\eps\to 0$.
Although these eigenfunctions remain orthogonal for all $\eps$ in the weighted space $L_2(\rho_\eps,\Omega)$, they make an infinitely small angle between them in $L_2(\Omega)$ with the standard norm, and stick together at the limit. In particular, it leads to the loss of completeness in $L_2(\Omega)$ for the limit eigenfunction collection.
Interestingly enough, however, the plane  $\pi_{\eps, j}$ that is the span of $u_{\eps, j}^{(-)}$ and $u_{\eps, j}^{(+)}\}$  has regular  behaviour as $\eps\to 0$.
The limit position of $\pi_{\eps, j}$ is the 2-dimensional space $\pi_j$ spanned by the functions $\beta_{0,j}\cdot\bV$ and $v_{*,j}$.

We actually have an example of singular perturbations in which the completeness property of  perturbed eigenfunction collection passes in some sense into the completeness of generalized eigenfunctions  of the limit non-self-adjoint operator. Although we didn't justify the asymptotics of eigenfunctions of  \eqref{PEproblem}, we can formally state that the non-self-adjoint operator $\cP$  contains all the information about the asymptotic behaviour of eigenvalues of the perturbed problem:
\begin{itemize}
  \item[$\circ$]  the spectrum of  $\cP$  is a limit set for the spectra $\sigma(T_\eps)$ as $\eps\to 0$;
  \item[$\circ$]  knowing the multiplicity of eigenvalues of  $\cP$, we can divide $\sigma(T_\eps)$ into finite or infinite subsets of eigenvalues with the same limits as $\eps\to 0$;
  \item[$\circ$] in the case $\lambda\in \sigma(A)\cap\sigma(B)$, the dimension of space $H_\lambda$ indicates how many eigenvalues of $T_\eps$ possess the half-integer power asymptotics;
  \item[$\circ$] the Jordan chains of $\cP$ are involved in the quasimodes of $T_\eps$ (in the formal asymptotics of  eigenfunctions of $T_\eps$).
\end{itemize}

\begin{figure}[t]
  \centering
  \includegraphics[scale=0.9]{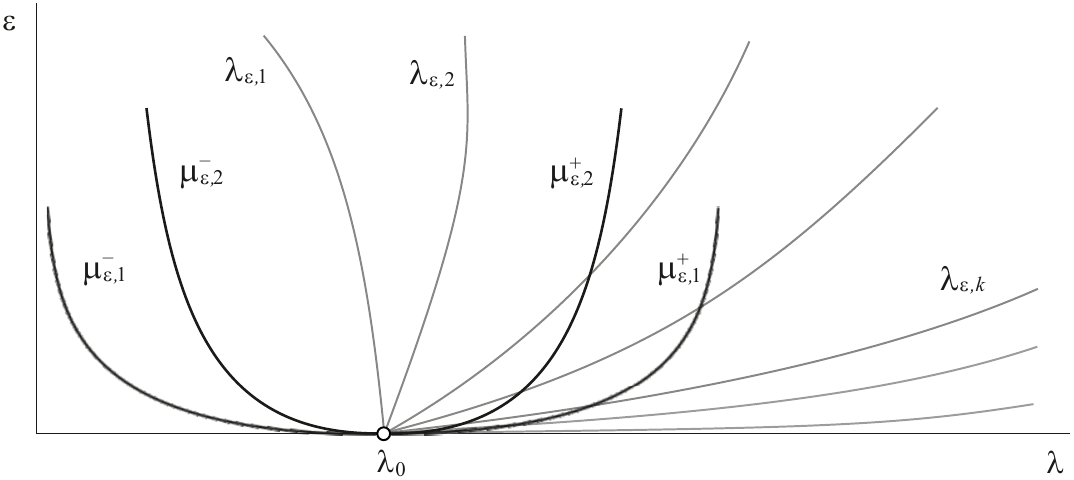}\\
  \caption{Bifurcation of the eigenvalue $\lambda_0\in \sigma(A)\cap\sigma(B)$.}\label{FigBifurcation}
\end{figure}

As we pointed out in the introduction,  the problem
\begin{equation}\label{PEproblemM3}
  - \Delta u_\eps = \lambda^\eps \rho_\eps u_\eps\quad\text{in }\Omega, \qquad  u_\eps=0 \quad \text{on } \partial\Omega
\end{equation}
with the perturbed density $\rho_\eps(x;m)=\rho(x)$ for $x\in \Omega\setminus \omega_\eps$ and $\rho_\eps(x;m)=\eps^{-m}q_\eps(x)$ for $x\in \omega_\eps$,
the Dirichlet boundary condition and  without the potential $a$ in the differential equation has been considered in \cite{GolovatyGomezLoboPerez2004}. The variational methods have been applied  to study the asymptotic behaviour of eigenvalues and eigenfunctions. In this case  the problem can be realized as the family of bounded self-adjoint operators $A_\eps$ in  $\mathring{W}^1_2(\Omega)=\{v\in W^1_2(\Omega)\colon u_\eps=0 \text{ on } \partial\Omega\}$. Due to the second representation theorem~\cite[Theorem~VI.2.23]{Kato},  $A_\eps$ is defined by the identity $\frak{a}(A_\eps u,\phi)=(\rho_\eps u,\phi)_{L_2(\Omega)}$ for all $\phi\in \mathring{W}^1_2(\Omega)$, because the Dirichlet form $\frak{a}(u,\phi)=\int_\Omega\nabla u\cdot \nabla\phi\,dx$ is a scalar product in $\mathring{W}^1_2(\Omega)$. Problem \eqref{PEproblemM3} can be written in the form $A_\eps u_\eps=\lambda_\eps^{-1} u_\eps$. For $m>1$ the operators $A_\eps$ diverge as $\eps\to 0$, and it has been shown that $\|A_\eps\|\leq c\eps^{1-m}$ with a constant $c$ being independent of~$\eps$.

The complete asymptotic analysis of  eigenvalues $\lambda_\eps$ of \eqref{PEproblemM3}  has been  carried out for $m=3$. We have proved  that the spectrum  consists of a countable number of infinite series of eigenvalues with the asymptotics
\begin{equation}\label{AsymptoticsM3}
\lambda^\eps_{kj}=\mu_k\eps+\nu_{kj}\eps^2+o(\eps^2)\quad \text{as } \eps\to 0, \quad k,j\in \mathbb{N},
\end{equation}
where $\mu_k$ is an infinite-fold eigenvalue of $B$ and $\{\nu_{kj}\}_{j=1}^\infty$ is a spectrum of pseudodifferential operator $N_{\mu_k}$ acting in $L_2(\gamma)$ (see formula \eqref{OperatorNlambda}). The operator $N_{\mu_k}$ is a self-adjoint realization of a boundary value problem for the Laplace operator in $\Omega$ with coupling conditions on $\gamma$.

Since the first eigenvalue $\mu_1$ of $B$ is equal to zero, the low-lying eigenvalues admit the asymptotics $\lambda^\eps_{1j}=\eps^2(\nu_{1j}+o(1))$, $\eps\to 0$. In this case, $\{\nu_{1j}\}_{j=1}^\infty$ are  eigenvalues of the spectral problem $-\Delta u=\nu q_0\delta_\gamma u$ in $\Omega$, $u=0$ on $\partial\Omega$,
where $q_0=\int_{-1}^1q(r)\,dr$ and $\delta_\gamma$  is a distribution in $\mathcal{D}'(\Omega)$ acting as $\langle \delta_\gamma, \phi\rangle=\int_\gamma \phi\,d \gamma$ for any $\phi\in C^\infty(\Omega)$.
The problem describes the eigenvibrations of a  membrane with the mass density $\rho_0(x)=q_0\delta_\gamma(x)$, i.e., the whole mass of the membrane is supported on $\gamma$ and the rest part of vibrating system is massless. Such series of infinitesimal eigenvalues with the asymptotics   $\lambda^\eps_{1j}=\eps^{m-1}(\nu_{1j}+o(1))$, $\eps\to 0$,  exists for any $m>1$. Hence, the estimate $\|A_\eps\|\leq c\eps^{1-m}$ is sharp. The similar result is obtained in Theorems \ref{ThmCaseII} and \ref{ThmCaseIII}  for problem \eqref{PEproblem} when  $m=2$ (see Corollary~\ref{CorLFV}).

In \cite{GolovatyGomezLoboPerez2004} some results on the convergence of eigenfunctions have been also obtained.  Interestingly enough, the complete asymptotic description of $\lambda^\eps$ and $u_\eps$ in \eqref{PEproblemM3} includes not only the operator $B$ but also the operator $A$ which is associated with the  problem $-\Delta u= \lambda u$ in $\Omega\setminus\gamma$,  $u=0$ on $\partial\Omega$ and $\gamma$ in this case. In view of \eqref{AsymptoticsM3}, any point of the positive spectral half-line is an accumulation point for eigenvalues $\lambda^\eps$, i.e., for each $\lambda>0$ there exists a sequence $\{\lambda^{\eps_i}\}_{i=1}^\infty$ of eigenvalues such that $\lambda^{\eps_i}\to \lambda$ as $\eps_i\to 0$.  We have proved that only the points of $\sigma(A)$ can be approximated by  eigenvalues $\lambda^{\eps_i}$ so that the corresponding eigenfunctions $u_{\eps_i}$ converge to  nontrivial limits in $L^2(\Omega)$. These limits are eigenfunctions of $A$ corresponding to $\lambda$.

\begin{cor}\label{CorLFV}
Problem \eqref{PEproblem} has a series of small eigenvalues that admit the asymptotics $\lambda_{\eps,k}=\eps\nu_k+O(\eps^2)$ as  $\eps\to 0$,
where $\nu_k$  are eigenvalues of the problem
\begin{equation}\label{ProblemWithDeltaM2}
        -\Delta v+av=0  \text{ in }\Omega\setminus\gamma,\quad  \ell v=0\text{ on } \partial\Omega,\quad [v]_\gamma=0, \;\; [\partial_rv]_\gamma+\nu q_0 v=0 \text{ on } \gamma
\end{equation}
with the spectral parameter $\nu$ in the coupling conditions. In addition, if $0\in\sigma(A)$, then \eqref{PEproblem} has also a finite number of small eigenvalues with asymptotics
\begin{equation*}
\mu_{\eps,j}^\pm=\pm\omega_j\sqrt{\eps}+O(\eps) \text{ as } \eps\to 0, \quad j=1,\dots,d.
\end{equation*}
\end{cor}
\begin{proof}
The existence of such eigenvalues follows from Theorems \ref{ThmCaseII} and \ref{ThmCaseIII} along with the observation that $\lambda_0=0$ belongs to $\sigma(B)$ (see Remark~\ref{remZeroP}). It remains to derive the eigenvalue problem for $\nu_k$.
Let us look at  \eqref{P0EqV}--\eqref{P0Coupling} and put $\lambda=0$:
\begin{gather*}
        -\Delta v+av=0 \;\;\text{in }\Omega\setminus\gamma,\quad \ell v=0\;\;\text{on } \partial\Omega,\quad v^-=w(\,\cdot\,,-1),\quad v^+=w(\,\cdot\,,1),\;\; \text{on }\gamma,\\
         \partial^2_n w=0   \:\text{ in } \omega,\quad
        \partial_n w(\,\cdot\,,-1)=0,\;\: \partial_n w(\,\cdot\,,1)=0,
\end{gather*}
We see that $w=b(s)$  and  hence  $v^-=v^+$, i.e., $v$ is continuous on $\gamma$.  Since $y=1$ in this case, condition \eqref{CompactibilityW1} for $v_0=v$  reads as $\partial_r v^+-\partial_r v^-+ \lambda_1 q_0 b=0$ on $\gamma$. To complete the proof we replace $\lambda_1$ with $\nu$ and recall that $b=v|_{\gamma}$.
Note that problem \eqref{ProblemWithDeltaM2} can be written in the form
\begin{equation*}
        -\Delta v+av=\nu q_0\delta_\gamma v  \text{ in }\Omega,\quad  \ell v=0\text{ on } \partial\Omega.
\end{equation*}
Hence, $\nu_k$  are eigenvalues of a membrane with the mass density $\rho_0(x)=q_0\delta_\gamma(x)$ supported by the curve $\gamma$.
\end{proof}

\section{Asymptotics of eigenvalues  in the case $\lambda_0\in\sigma(A)\setminus\sigma(B)$ }\label{SectCaseI}
For the sake of completeness, we briefly discuss the perturbation of eigenvalue $\lambda_0$ of $\cP$ which does not belong to $\sigma(B)$.
In view of part \textit{(i)} of Theorem~\ref{TheoremSpectrumStructure}, $\lambda_0$ is an  eigenvalue of finite multiplicity. Also, we have $v_0=\alpha_0\cdot\bV$ and $w_0=0$ in asymptotics \eqref{AsympExUI}, where $\alpha_0\in \Real^K\setminus\{0\}$.
Then \eqref{EqW1new} and \eqref{NeumannCondW1} imply
\begin{equation*}
  -\partial^2_n w_1=\lambda_0 qw_1 \;\;\text{in } \omega,\quad
  \partial_n w_1(\,\cdot\,,-1)=\alpha_0\cdot\partial_r \bV^-,\;\;  \partial_n w_1(\,\cdot\,,1)=\alpha_0\cdot\partial_r \bV^+.
\end{equation*}
Since $\lambda_0\not\in \sigma(B)$, there exists a unique solution of the problem. Suppose the functions $W_k$, $k=1,\dots,K$, solve the problems
\begin{equation*}
  -\partial^2_n w=\lambda_0 qw \;\;\text{in } \omega,\quad
  \partial_n w(\,\cdot\,,-1)=\partial_r V_k^-,\;\;  \partial_n w_1(\,\cdot\,,1)=\partial_r V_k^+,
\end{equation*}
and $\bW=(W_1,\dots,W_K)$. Thus we have $w_1=\alpha_0\cdot\bW$. Next, we can rewrite \eqref{EqV1} and \eqref{CouplingV1W1} in the form
\begin{align}\label{v1finite}
        &-\Delta v_1+av_1=\lambda_0 \rho v_1+\lambda_1 \rho \,\alpha_0\cdot\bV\;\;\text{in }\Omega\setminus\gamma,\quad \ell v_1=0\;\;\text{on } \partial\Omega,\\\label{condv1finite}
         &\quad v_1^-=\alpha_0\cdot(\bW(\,\cdot\,,-1)+\partial_r \bV^-),\quad    v_1^+=\alpha_0\cdot(\bW(\,\cdot\,,1)-\partial_r \bV^+).
\end{align}
In general, the problem is unsolvable, because  $\lambda_0$ belongs to the spectrum of $A$ which is the direct sum of $A_-$ and $A_+$. To achieve the solvability one needs to choose
 proper vectors $\alpha_0$ along with the parameter $\lambda_1$. Multiplying the equation in \eqref{v1finite} by $V_1,\dots,V_K$ in turn and integrating by parts
twice in view of the boundary conditions \eqref{condv1finite} yield the spectral matrix equation $R\alpha_0=\lambda_1\alpha_0$, where the matrix $R=(r_{ij})_{i,j=1}^K$ has the entries
\begin{equation*}
  r_{ij}=\int_\gamma \left(
  (W_i(\,\cdot\,,1)-\partial_r V_i^+)\partial_r V_j^+
  +(W_i(\,\cdot\,,-1)+\partial_r V_i^-)\partial_r V_j^-
  \right)\,d\gamma.
\end{equation*}
The matrix $R$ is symmetric, because  it is easy to check that
\begin{equation*}
\int_\gamma W_i(\,\cdot\,,\pm1)\partial_r V_j^\pm\,d\gamma=\int_\gamma W_j(\,\cdot\,,\pm1)\partial_r V_i^\pm\,d\gamma.
\end{equation*}
Assume that $R$ has $K$ simple eigenvalues $\lambda_1^{(1)},\dots,\lambda_1^{(K)}$ with the corresponding eigenvalues $\alpha_0^{(1)},\dots,\alpha_0^{(K)}$.
For any pair $(\lambda_1^{(k)}, \alpha_0^{(k)})$, we can solve \eqref{v1finite}, \eqref{condv1finite} and find  $v_1^{(k)}=\hat{v}_1^{(k)}+\alpha_1^{(k)}\cdot\bV$ up to the vector $\alpha_1^{(k)}\in\Real^K$. We continue in this fashion obtaining $K$ different quasimodes with high enough accuracy of  the operator $T_\eps$  that approximate the part of spectrum lying in a vicinity of  $\lambda_0$.

\begin{thm}
  Suppose that $\lambda_0$ is an eigenvalue of the limit operator $\cP$ such that $\lambda_0\in\sigma(A)\setminus \sigma(B)$.
Assume $\lambda_0$  has multiplicity $K$ and  the matrix $R$ possesses the simple eigenvalues $\lambda_1^{(1)},\dots,\lambda_1^{(K)}$.
Then there exist $K$ eigenvalues in the spectrum of $T_\eps$ that
converge to $\lambda_0$  and admit the asymptotics $\lambda_{\eps,k}=\lambda_0+\eps\lambda_1^{(k)}+O(\eps^2)$ as $\eps\to 0$, $k=1,\dots,K$.
\end{thm}


\end{document}